\documentclass[11pt]{amsart}

\usepackage{amssymb}
\usepackage{amsthm, bm}
\usepackage{amsmath}
\usepackage{graphicx}
\usepackage{comment}
\usepackage[usenames,dvipsnames]{xcolor}
\usepackage[margin = 1in]{geometry}
\usepackage{enumerate}
\usepackage[toc,page]{appendix}
\usepackage{lineno}
\usepackage{color}
\usepackage{hyperref}
\usepackage{marginnote}
\usepackage{todonotes}

\definecolor{mygreen}{rgb}{0.1,0.75,0.2}

 \newtheorem{thm}{Theorem}[section]
 \newtheorem{cor}[thm]{Corollary}
 \newtheorem{lem}[thm]{Lemma}
 \newtheorem{prop}[thm]{Proposition}

 \theoremstyle{definition}
 \newtheorem{defn}{Definition}

 \theoremstyle{remark}
 \newtheorem{rem}{Remark}

 \numberwithin{equation}{section}

\newcommand{\pt}{\partial}
\newcommand{\eps}{\varepsilon}

\newcommand{\sij}{{ij}}
\newcommand{\bR}{{\mathbb{R}}}
\newcommand{\ud}{\,\mathrm{d}}
\newcommand{\8}{\infty}
\newcommand{\F}{\mathcal{F}}
\newcommand{\cS}{\mathcal{S}}
\newcommand{\cF}{\mathcal{F}}
\newcommand{\cL}{\mathcal{L}}
\newcommand{\cE}{\mathcal{E}}
\newcommand{\cnu}{\bm v }

\newcommand{\ptf}{\mathcal{L}}
\newcommand{\el}{E_{\mathrm{els}}}
\newcommand{\elc}{\mathcal{C}_{\mathrm{els}}}
\newcommand{\hel}{\hat{E}_{\mathrm{els}}(\bm \varphi \vert \mathbf u)}
\newcommand{\helB}{\hat{E}_{\mathrm{els}}(\bm \varphi \vert \mathbf u; B_R)}

\newcommand{\heg}{\hat{E}_{\Gamma}( \bm \varphi \vert  \mathbf{u})}

\newcommand{\hegg}{\hat{E}_{\Gamma_e}(\bm \varphi \vert  \mathbf{u})}
\newcommand{\het}{\hat{E}(\bm \varphi \vert  \mathbf u)}
\newcommand{\hetB}{\hat{E}(\bm \varphi \vert  \mathbf u; B_R)}

\newcommand{\epsv}{\eps_{\varphi}}

\newcommand{\sigvv}{\sigma_{\varphi}}

\begin{document}

\title[Peierls-Nabarro model for curved dislocation]{Existence and uniqueness of bounded stable solutions to the Peierls-Nabarro model for curved dislocations}

\author{Hongjie Dong}
\author{Yuan Gao}

\address{Division of Applied Mathematics, Brown University, 182 George Street, Providence, RI, 02912, USA}
\email{hongjie\_dong@brown.edu}

\address{Department of Mathematics, Duke University, Durham NC 27708, USA}
\email{yuangao@math.duke.edu}

\date{\today}

\begin{abstract}
We study the well-posedness of the vector-field Peierls-Nabarro model for curved dislocations with a double well potential and a bi-states limit at far field. Using the Dirichlet to Neumann map, the 3D Peierls-Nabarro model is reduced to a nonlocal scalar Ginzburg-Landau equation. We derive an integral formulation of the nonlocal operator, whose kernel is anisotropic and positive when Poisson's ratio $\nu\in(-\frac12, \frac13)$. We then prove that any bounded stable solutions to this nonlocal scalar Ginzburg-Landau equation has a 1D profile, which corresponds to the PDE version of flatness result for minimal surfaces with anisotropic nonlocal perimeter.
Based on this, we finally obtain that steady states to the nonlocal scalar equation, as well as the original Peierls-Nabarro model, can be characterized as a one-parameter family of straight dislocation solutions to a rescaled 1D Ginzburg-Landau equation with the half Laplacian.
\end{abstract}

\maketitle
\section{Introduction}
Materials defects such as dislocations are important structures in crystalline materials and play essential roles in the study of
plastic and mechanical behaviors of materials. {Along} the  dislocation line, there is a small region (called the dislocation core region) of heavily
distorted atomistic structures with shear displacement jump across a slip plane, denoted {by}
$$
\Gamma:=\{(x_1,x_2,x_3):~x_3=0\}.
$$
The classical dislocation theory \cite{HL}  regards the dislocation core as a
singular point so that the solution can be solved explicitly based on the linear elasticity theory, which, however, is not able to unveil detailed core structure of dislocations. Instead,  the Peierls-Nabarro (PN) model introduced by \textsc{Peierls and Nabarro} \cite{Peierls, Nabarro}  is
a multiscale continuum model for displacement $\mathbf{u}=(u_1, u_2, u_3)$ that incorporates the atomistic effect by introducing a nonlinear
potential describing the atomistic misfit interaction across the slip plane $\Gamma$ of the dislocation.  More precisely, assume two elastic continua $x_3>0$ and $x_3<0$ are connected by a nonlinear atomistic potential
$\gamma$ depending on shear displacement jump
$$
([u_1], [u_2]):=(u_1(x_1,  x_2, 0^+)-u_1(x_1,  x_2, 0^-), ~u_3(x_1,  x_2, 0^+)-u_3(x_1,  x_2, 0^-))
$$
across the slip plane $\Gamma$. {Although} the total jump increment is determined by the magnitude of Burgers vector, the true spread of the jump increment $([u_1], [u_2])$  is determined globally by the whole system, particularly for curved dislocation with variant  orientations. Given the magnitude of the Burgers vector, which means that given the bi-states boundary conditions at far field (see \eqref{BC_far}), the problem turns out to be a minimization problem of the total energy
$$
E(\mathbf{u}):=E_{\text{els}}(\mathbf{u})+E_{\text{mis}}(\mathbf{u});
$$
see detailed definitions {of the} elastic energy $E_{\text{els}}(\mathbf{u})$ and the misfit energy $E_{\text{mis}}(\mathbf{u})$ in Section \ref{sec2.1}. The resulting Euler-Lagrange equations for {the} vector-field $\mathbf{u}$ is a L\'ame system with {a} nonlinear boundary condition; see  \eqref{maineq}.

For a straight dislocation with uniform displacement in the $x_2$ direction, the 2D L\'ame system with the nonlinear boundary condition can be reduced to a nonlocal equation (also known as nonlocal Ginzburg-Landau equation with double-well potential $\gamma$)
\begin{equation}\label{old}
(-\Delta)^{\frac12}u_1 (x_1)= - \gamma'(u_1(x_1)), \quad x_1\in \bR
\end{equation}
with {the} bi-states at far field $u_1(\pm\8)=\pm1$.
For a special sinusoidal misfit potential reflecting phenomenologically lattice periodicity $\gamma(u_1)=\frac{1}{\pi^2}(1+ \cos(\pi u_1))$, with certain physical constants for computational simplicity, the solution can be solved explicitly \cite{HL, Xiang_2006} with shear displacement
\begin{equation*}
u_1(x_1)= \frac{2}{\pi} \arctan(x_1) \sim \pm1-\frac{2}{\pi x_1}, \quad \text{ as }x_1\to \pm \8.
\end{equation*}
Equation \eqref{old}, as well as
the corresponding scalar displacement $\tilde{u}(x_1,x_2)$, as the harmonic extension of $u_1(x_1,0)$, is well studied recently at rigorous mathematical level.
For a general misfit potential $\gamma$ with $C^{2,\alpha}$ regularity \cite{CS05}, \textsc{Cabr\'e and Sol\`a-Morales} established the existence  (and the uniqueness up to translations) of  monotonic solutions with  the sharp decay rate $\frac1{x_1}$ for the bistable profile. They also proved the bistable profile is a local minimizer with respect   to perturbations with compact support for the total energy
$$
E(\tilde{u})=\frac 1 2\int |\nabla \tilde{u}|^2\ud x_1\ud x_2+\int_\Gamma \gamma(\tilde{u})\ud x_1
$$
of the scalar model using {the} harmonic extension.  Without using the harmonic extension,  \textsc{Palatucci, Savin, and Valdinoci}  directly worked on the nonlocal equation $(-\Delta)^{\frac12}\tilde{u}|_\Gamma=-\gamma'(\tilde{u})$ on $\Gamma$ and improved  the global minimizer  result  by proving quantitative growth estimates of the total energy \cite{PSV13}. {In} \cite{GLLX19}, {the authors} established rigorously the connection between the vector-field 2D L\'ame system and the {reduced} equation \eqref{old} at both the equation and energy level by considering a perturbed energy.    Besides,  for the De Giorgi-type hyperplane conjecture for stable  solutions to \eqref{old} (also known as De Giorgi-type hyperplane conjecture for {the Laplace} equation with nonlinear boundary reaction),  {it was proved in} \cite{CS05} that in 2D, bounded stable solutions {have} 1D {profiles}.
For a general nonlocal Ginzburg-Landau  equation
$$
(-\Delta )^{s} \tilde{u} = -\gamma'(\tilde{u}), \quad x\in \mathbb{R}^{d},
$$
 we refer {the readers} to some recent results for $d=2, 0<s<1$ by \textsc{Cabr\'e and Sire} \cite{CaSi}; for $d=3, s=\frac12$  by \textsc{Cabr\'e and Cinti}   \cite{cabre2010energy}; for $d=3,\, \frac12<s<1$  by \textsc{Cabr\'e and Cinti}   \cite{cabre2014sharp}  and  for $d=3, s=\frac12$  by \textsc{Figalli and Serra} \cite{FS20}; and related energy estimates for $0<s<1$ by \textsc{Gui and Li} \cite{Gui19} and flatness result for $0<s<\frac12$ by \textsc{Dipierro, Serra and Valdinoci} \cite{dipierro2016nonlocal, dipierro2016improvement} and for $\frac12\leq s<1$ by \textsc{Savin} \cite{savin2018rigidity}.

However, to our best knowledge, {so far} there is no result for the 3D vector-field system \eqref{maineq}, which cannot be treated as an analogue scalar model above.  In fact, the vector-field displacement is essential to determine
the long-range elastic interaction associated with dislocations and partial separation within {the} dislocation
core. We are especially interested in the curved dislocation \cite{Xiang_Wei_Ming_E_2008}, which is the most common case,  and their properties are anisotropic in space and depend on the orientations of dislocations.  The main goal of this paper is to study in which cases, the steady state (equilibrium) of the PN model \eqref{maineq} has to be a straight dislocation. This also establishes the foundation of further researches on the dynamics and long-time behaviors of curved dislocations.

There are in general two strategies to study the {solutions} to the full system \eqref{maineq}. One is to study the local vector-field system with nonlinear Neumann boundary conditions. However, the challenges come from the lack of maximum principle and the lack of the compactness in unbounded domain. {The other} strategy is to reduce the 3D full system to a nonlocal 2D problem using Dirichlet to Neumann map, which, in the curved dislocation case, is still a coupled nonlocal system; see \eqref{eq2D}. Under the assumption that the misfit potential $\gamma$ depends only on the shear jump displacement $[u_1]$. We will further {reduce} it to a scalar nonlocal equation (see \eqref{maineq1D}) and study the resulting nonlocal operator, which shows anisotropic property in different directions. The kernel of the new nonlocal operator is still homogeneous but anisotropic; see Propositions \ref{prop_kernel1} and \ref{prop_kernel2}.  Especially, the kernel remains positive only for Poisson's ratio $\nu$ in the range $(-\frac12,\frac13).$

With a positive anisotropic kernel, a natural question is the existence and uniqueness of solutions to the nonlocal equation. Since the straight dislocation is {a} special solution to the full system, we are particularly interested in the {characterization} of the solutions, i.e., if the misfit potential depends only on $[u_1]$, whether the straight solution is the {only} stable solution to the full system \eqref{maineq}. We will  follow the idea in \cite{CSV19}, which proves quantitative flatness estimates for {the} stable sets with nonlocal perimeters (see also \cite{FS20, Gui19, dipierro2016nonlocal, savin2018rigidity} for PDE version with fractional Laplacian), to first show {that any} bounded stable solution to \eqref{maineq1D} {has} a 1D profile; see Theorem \ref{thm-1d}.
As a consequence, {all} the {solutions} to \eqref{maineq1D} as well as \eqref{maineq} {can be characterized as} a rotation of straight dislocation. This is analogue to {the} flatness result for the isotropic case with the half Laplacian. However, for the general case when the misfit potential depends both on $[u_1]$ and $[u_2]$, the characterization of solutions to the coupled nonlocal Ginzburg-Landau system \eqref{eq2D} remains open.

The paper will be organized as follows. In Section \ref{sec2}, we propose the governing equations for the full vector-field system and then reduce it {to} a nonlocal equation \eqref{maineq1D} by the Dirichlet to Neumann map. In Section  \ref{sec3}, we derive the integral formulation of the new nonlocal operator and study the positivity of the resulting anisotropic kernel. In Section \ref{1Dprop}, we prove {that any} bounded stable solution to the reduced nonlocal equation \eqref{maineq1D} {has} a 1D monotone profile and is given by a rotation of straight dislocation. The derivation of {the} Euler-Lagrange equation and the Dirichlet to Neumann map will be given in Appendices \ref{appA} and \ref{appB}, respectively.

\section{Full system and reduced nonlocal system by the Dirichlet to Neumann map}\label{sec2}
In this section, we will first derive the Euler-Lagrange equation for the PN model, which is a minimization problem of the total energy consisting of {the} elastic energy and the misfit energy induced by a dislocation; see Section \ref{sec2.1}. Then we will derive the reduced nonlocal systems/equation by the Dirichlet to Neumann map in Section \ref{sec2.2}.
\subsection{Vector-field full system with nonlinear boundary condition}\label{sec2.1}
In the PN model, the two half spaces separated by the slip plane $\Gamma=\{(x_1,x_2,x_3); x_3=0\}$ of the dislocation are assumed to be linear elastic continua, and the two half spaces are connected by a nonlinear potential
energy across the slip plane that incorporates  atomistic interactions. Let us first clarify the total energy, which is indeed infinite in $\bR^3$ due to the presence of a dislocation \cite{GLLX19}, and then derive the Euler-Lagrange equation by regarding the solution as a local minimizer of the total energy.

  Let $\mathbf u=(u_1,u_2, u_3)$ be the displacement vector. The total energy $E(\mathbf u)$  of the whole system  is
\begin{equation}\label{E2.2}
E(\mathbf u):= E_\mathrm{els}(\mathbf u)+E_\mathrm{mis}(\mathbf u).
\end{equation}
Let $G>0$ be the shear modulus and  $\nu\in[-1,\frac12]$ be Poisson's ratio.
The first term in the total energy in Eq.~\eqref{E2.2} is the elastic energy in the two half spaces defined as $$E_{\text{els}} = \int_{\mathbb{R}^3\backslash \Gamma} \frac12 \eps: \sigma \ud x= \int_{\mathbb{R}^3\backslash \Gamma} \frac12 \eps_{ij} \sigma_{ij} \ud x,$$
where $\varepsilon$ is the strain tensor
$$
\eps_\sij=\frac{1}{2}(\pt_j u_i+\pt_i u_j) \quad \text{ for }i,j=1,2,3,\quad \partial_i:=\frac{\partial}{\partial x_i},
$$
and $\sigma$ is the stress tensor
$$
\sigma_\sij=2G \eps_\sij+\frac{2\nu G}{1-2\nu} \eps_{kk}\delta_{ij} \quad \text{ for }i,j=1,2,3.
$$
 Here  $\delta_{ij}=1$ if $i=j$ and $0$ otherwise. We also used the Einstein summation convention that
$$
\eps_{kk}=\sum_{k=1}^3 \eps_{kk}\quad \text{and}\quad \sigma_{ij}\varepsilon_{ij}=\sum_{i,j=1}^3 \sigma_{ij}\varepsilon_{ij}.
$$
The second term in the total energy in Eq.~\eqref{E2.2} is the misfit energy across the slip plane due to nonlinear atomistic interactions
\begin{equation*}
 {E_\mathrm{mis}}(\mathbf u):=\int_{\Gamma} \gamma(u_1^+-u_1^-, u_2^+-u_2^-) \ud \Gamma =\int_{\Gamma} W (u_1^+, u_2^+) \ud \Gamma,
 \end{equation*}
 where $u_i^\pm=u_i^+(x_1, x_2, 0^\pm)$ for $i=1,2$.
 For the analysis of the PN model for an edge dislocation in this paper, we assume that the nonlinear potential $W \in C_b^{2, \alpha}(\mathbb{R}^2; \bR)$ for some $\alpha\in(0,1)$. In practice, $W$ will be a periodic potential indicating {the} periodic lattice structure of the materials with several minimums, for instance $W(v_1,v_2)=\cos v_1 + \sin v_2,$ and will be specific later.

The equilibrium structure of a general curved dislocation is obtained by minimizing the total energy in Eq.~\eqref{E2.2} subject to the boundary condition at the slip plane
\begin{equation}\label{BC}
\begin{aligned}
u_1^+(x_1, x_2, 0^+)= - u_1^-(x_1,x_2, 0^- ),\\
u_2^+(x_1, x_2, 0^+)=  -u_2^-(x_1,x_2, 0^-),\\
u_3^+(x_1, x_2, 0^+)= -u_3^-(x_1,x_2, 0^-).
\end{aligned}
\end{equation}
To focus on nontrivial solutions indicating the presence of a curved dislocation, we consider the following bi-states far field boundary condition
 for $u_1$,
 \begin{equation}\label{BC_far}
 u_1^+(\pm\8,x_2, 0^+)=\pm 1 \quad \text{ for any } x_2\in \mathbb{R},
 \end{equation}
 where we chose certain magnitude of the Burgers vector for simplicity.

However, due to the slow decay rate of the strain tensor $\eps$, we have the same issue with straight dislocation as in \cite{CS05, PSV13, GLLX19}, i.e., the elastic energy is infinite.   Whenever the total energy is infinite, we define the energy minimizer in {the} perturbed sense with respect to a perturbation with compact support.
To be precise, we define the perturbed elastic energy of $\mathbf u$ with respect to any perturbation fields $ \bm\varphi  \in C^\infty(\mathbb{R}^3\backslash \Gamma; \mathbb{R}^3)$  and $ \bm\varphi $ has compact support in some $B_R\subset \bR^3$  as
\begin{equation*}
\begin{aligned}
\hel:=& \int_{\mathbb{R}^3\backslash \Gamma} \frac{1}{2} (\eps_u + \eps_\varphi): (\sigma_u + \sigvv)  -\frac{1}{2} \varepsilon_u:\sigma_u~\ud x\\
    =& \int_{\mathbb{R}^3\backslash \Gamma} \frac{1}{2}  [(\epsv)_{ij}(\sigvv)_{ij}+(\epsv)_{ij}(\sigma_u)_{ij}+ (\eps_u)_{ij}(\sigvv)_{ij} ]   ~\ud x\\
    =& \el(\bm\varphi ) + \elc(\mathbf u, \bm\varphi ),
\end{aligned}
\end{equation*}
where the cross term
\begin{equation*}
  \elc(\mathbf u, \bm\varphi ):= \int_{\mathbb{R}^3\backslash \Gamma} \frac{1}{2}  (\epsv:\sigma_u+ \eps_u:\sigvv )   ~\ud x= \int_{\mathbb{R}^3\backslash \Gamma} \frac{1}{2}  [(\epsv)_{ij}(\sigma_u)_{ij}+ (\eps_u)_{ij}(\sigvv)_{ij} ]   ~\ud x,
\end{equation*}
{and} $\eps_u, \sigma_u$ and $\epsv, \sigvv$ are the strain and stress tensors corresponding to $\mathbf u$ and $\bm\varphi$, respectively.
Then the perturbed total energy is defined as
\begin{equation}\label{perEtotal}
\het:= \hel+\int_\Gamma W(u_1+\varphi_1)-W(u_1) \ud x
\end{equation}
and the energy minimizer is defined as $\mathbf u$ such that $\het \geq 0$ {for any $\bm\varphi$ with compact support}.
\begin{rem}
Since $\mathbf u$ and $\bm \varphi$ coincide outside $B_R$, we always know {that} $\hel$ is equivalent to the local perturbed elastic energy
\begin{equation*}
\helB:=\int_{B_R\backslash \Gamma} \frac{1}{2} (\eps_u + \eps_\varphi): (\sigma_u + \sigvv)  -\frac{1}{2} \varepsilon_u:\sigma_u~\ud x= E_{\rm els}(\mathbf u + \bm \varphi; B_R)-E_{\rm els}(\mathbf u; B_R),
\end{equation*}
and $\het$ is equivalent to
\begin{equation*}
\hetB:=\helB+\int_{B_R\cap\Gamma} W(u_1+\varphi_1)-W(u_1) \ud x=E(\mathbf u + \bm \varphi; B_R)-E (\mathbf u; B_R).
\end{equation*}
Therefore, we will follow the convention \cite{CS05, PSV13} that refers $\mathbf u$ as a local minimizer. In the remaining {part} of the paper, whenever we consider the equivalence of two infinite energy, it is understood in the perturbed sense \cite{GLLX19} or equivalently, in the local sense in any balls $B_R$.
\end{rem}

\begin{defn}\label{minimizer}
We call the function $\mathbf u$ a local minimizer of total energy $E$ if it satisfies
\begin{equation*}
E(\mathbf u + \bm \varphi; B_R)-E(\mathbf u;B_R)\geq 0
\end{equation*}
for any perturbation $ \bm\varphi  \in C^\infty(\mathbb{R}^3\backslash \Gamma; \mathbb{R}^3)$ supported in some $B_R$ satisfying
\begin{equation}\label{bcphi}
\begin{aligned}
\varphi_1^+(x_1, x_2, 0^+)&= - \varphi_1^-(x_1, x_2, 0^-),\\
\varphi_2^+(x_1, x_2, 0^+)&=  -\varphi_2^-(x_1, x_2, 0^-),\\
\varphi_3^+(x_1, x_2, 0^+)&=  \varphi_3^-(x_1, x_2, 0^-).
\end{aligned}
\end{equation}
\end{defn}

We have the following lemma for the Euler--Lagrange equation with respect to the total energy  $E(\mathbf u)$, which gives the governing equations for the vector-field full system. The proof of this lemma will be included in Appendix \ref{appA} for completeness.

\begin{lem}\label{Lem2.2}
Assume that
 $\mathbf u\in C^2(\mathbb{R}^3\backslash \Gamma; \bR^3) $ satisfying boundary conditions \eqref{BC} and \eqref{BC_far}  is a local  minimizer of the total energy $E$ in  the sense of Definition \ref{minimizer}. Then $\mathbf u$ satisfies the  Euler--Lagrange equation
\begin{equation}\label{maineq}
\begin{aligned}
&(1-2\nu)\Delta \mathbf u+\nabla( \nabla\cdot \mathbf u)=0 \quad {\rm in} \ \mathbb{R}^3\backslash\Gamma, \vspace{1ex}\\
&\sigma_{13}^+ +\sigma_{13}^-=\pt_{1}W (u_1^+, u_2^+) \quad {\rm on} \ \Gamma,\\
&\sigma_{23}^+ +\sigma_{23}^-=\pt_{2}W (u_1^+, u_2^+) \quad {\rm on} \ \Gamma,\\
&\sigma_{33}^+ =\sigma_{33}^- \quad {\rm on}\ \Gamma.
\end{aligned}
\end{equation}
\end{lem}

\subsection{Dirichlet to Neumann map and the reduced nonlocal problem}\label{sec2.2}
In this section, we first take the strategy which reduces the 3D vector-field full system to a  nonlocal system in $\bR^2$ using the Dirichlet to Neumann map. Then  we will focus on the case that misfit potential $W$ depends only on the shear jump of {the} first component of the displacement field, which allows us to further reduce the problem to a scalar nonlocal Ginzburg-Landau  equation in $\bR$.
\subsubsection{ Reduction of the 3D full system to a nonlocal 2D system}
First we give the following Dirichlet to Neumann  map such that the vector-field displacement $\mathbf u$ can be expressed by the Dirichlet values of $u_1, u_2$ on $\Gamma$. The proof of this lemma is standard and will be given in Appendix \ref{appB}.
\begin{lem}[Dirichlet to Neumann  map] \label{D2N}
Assume {that} $\mathbf u$ is the solution to \eqref{maineq} such that the Dirichlet value of $u_1, u_2$ on $\Gamma$ are in $\dot{H}^s$ for some $s\geq \frac12$. Then solution $\mathbf u$ can be determined uniquely by  $u_1|_{\Gamma}, u_2|_\Gamma$. Particularly,  $\sigma_{13}(x_1,x_2, 0^+)$ and $\sigma_{23}(x_1,x_2,0^+)$ can be expressed in the Fourier space
 \begin{equation}\label{curve}
\left(\begin{array}{c}
\hat{\sigma}_{13}(k)\\
\hat{\sigma}_{23}(k)
\end{array}\right) = -A\left(\begin{array}{c}
\hat{u}_{1}(k)\\
\hat{u}_{2}(k)
\end{array}\right) := 2G
\left(\begin{array}{c}
{\textstyle -\left(\frac{k^2_2}{|k|}+\frac{1}{(1-\nu)}\frac{k^2_1}{|k|}\right)\hat{u}_1(k)
    -\frac{\nu }{(1-\nu)}\frac{k_1k_2}{|k|}\hat{u}_2(k)}\\
{\textstyle -\frac{\nu }{(1-\nu)}\frac{k_1k_2}{|k|}\hat{u}_1(k)
    -\left(\frac{k^2_1}{|k|}+\frac{1}{(1-\nu)}\frac{k^2_2}{|k|}\right)\hat{u}_2(k)}
\end{array}\right),
 \end{equation}
where $k=(k_1,k_2)$ is the frequency vector,  $|k|=\sqrt{k^2_1+k^2_2}$, and $\nu\in[-1,\frac12]$ is Poisson's ratio.
\end{lem}
Without loss of generality, from now on, we  set shear modulus $G$ to be $\frac12$ and use the notation $u_1=u_1^+(x_1,x_2, 0^+),\, u_3=u_3^+(x_1,x_2, 0^+),\,\, (x_1,x_2)\in \Gamma,$ for simplicity.

From the Dirichlet to Neumann  map in Lemma \ref{D2N}, the nonlinearity is decoupled and we obtain a 2D nonlocal system
\begin{equation}\label{eq2D}
\left(\begin{array}{c}
-\sigma_{13}^+ -\sigma_{13}^-\\
-\sigma_{23}^+ -\sigma_{23}^-
\end{array}\right)=:\mathcal{A} \left( \begin{array}{c}
u_1\\
u_2
\end{array} \right)=  \left( \begin{array}{c}
-\pt_{1}W(u_1, u_2)\\
-\pt_2 W(u_1, u_2)
\end{array} \right)\,\, \text{ on } \Gamma,
\end{equation}
where the nonlocal operator $\mathcal{A}$ is expressed below in \eqref{2dopt}. For straight dislocation, we refer to \cite{GLLX19} for {details about} the equivalence between the full system  and  reduced 1D equation in terms of both solutions and energies. Below, we formally derive the equivalence of the energies for the full system \eqref{maineq} and the reduced problem \eqref{eq2D}.

Recall that
\begin{equation*}
{A} = \left(\begin{array}{cc}
 \frac{k^2_2}{|k|}+\frac{1}{(1-\nu)}\frac{k^2_1}{|k|}
    & \frac{\nu }{(1-\nu)}\frac{k_1k_2}{|k|}\\
 \frac{\nu }{(1-\nu)}\frac{k_1k_2}{|k|}
      & \frac{k^2_1}{|k|}+\frac{1}{(1-\nu)}\frac{k^2_2}{|k|}
\end{array}\right)=  \left(\begin{array}{cc}
|k|
    & 0\\
0
      & |k|
\end{array}\right) + \frac{\nu}{1-\nu}\left(\begin{array}{cc}
\frac{k^2_1}{|k|}
    &\frac{k_1k_2}{|k|}\\
 \frac{k_1k_2}{|k|}
      & \frac{k^2_2}{|k|}
\end{array}\right),
\end{equation*}
which is {positive definite} for the Poisson's ratio between $\nu\in(-1,\frac12)$.
For $x=(x_1,x_2), x'=(x_1',x_2')$, recall the Riesz potential in 2D is
\begin{equation*}
I_\alpha f (\mathbf{x}) := c \int_{\mathbb{R}^2} | x-x' |^{-2+\alpha} f(x') \ud x', \quad 0<\alpha<2
\end{equation*}
with the Fourier symbol $|k|^{-\alpha}$. Thus for $r:=\sqrt{x_1^2+x_2^2}=|x|$,
\begin{align*}
\mathcal{F}^{-1} (\frac{k_i k_j}{|k|} \hat{f}) = \pt_{ij} \frac{1}{r} * f, \quad i,j=1,2,
\end{align*}
where $\F$ means the Fourier transformation.
We can rewrite $\mathcal{A}$ as
\begin{equation}\label{2dopt}
\mathcal{A}\left( \begin{array}{c}
u_1\\
u_2
\end{array} \right) = \int_{\mathbb{R}^2} G(x- x') \left( \begin{array}{c}
u_1(x)-u_1(x')\\
u_2(x)-u_2(x')
\end{array} \right) \ud x',
\end{equation}
where
\begin{equation*}
G(x) := \frac{1}{r^3} \left[  \frac{1-2\nu}{1-\nu} \left(\begin{array}{cc}
1
    &0\\
 0
      & 1
\end{array}\right) +
 \frac{3\nu}{1-\nu}
\left(\begin{array}{cc}
\frac{x_1^2}{r^2}
    &\frac{x_1x_2}{r^2}\\
 \frac{x_1 x_2}{r^2}
      & \frac{x_2^2}{r^2}
\end{array}\right)
  \right]= \frac{1}{|x|^3} \left( \frac{1-2\nu}{1-\nu}I + \frac{3\nu}{1-\nu} \frac{x}{|x|} \otimes \frac{x}{|x|} \right).
\end{equation*}
For $\nu\in(-1,\frac12)$, since $A$ is  positive defined, from Plancherel's equality, we have
\begin{equation*}
c_2 \|(u_1, u_2)\|^2_{\dot{H}^{\frac12}(\Gamma)} \leq \int_{\mathbb{R}^2}(u_1, u_2)^T \mathcal{A}  \left( \begin{array}{c}
u_1\\
u_2
\end{array} \right) \ud x \leq {C_2} \|(u_1, u_2)\|^2_{\dot{H}^{\frac12}(\Gamma)}.
\end{equation*}
Rigorously, the inequality shall be understood in perturbed sense; see \cite{GLLX19}.
Denote the reduced energy on $\Gamma$ as
\begin{equation}\label{energy2d-o}
E_\Gamma:= \frac 1 2\int_{\mathbb{R}^2}(u_1, u_2)^T \mathcal{A}  \left( \begin{array}{c}
u_1\\
u_2
\end{array} \right) \ud x+ \int_{\mathbb{R}^2} W(u_1, u_2) \ud x.
\end{equation}
Similar to \eqref{perEtotal}, we define the perturbed elastic energy of $\mathbf u$ on $\Gamma$ with respect to  the perturbation $\bm\varphi\in C_c^\8(\bR^3; \bR^3) $ as
\begin{equation*}
\begin{aligned}
\hegg:=&  \frac 1 2\int_\Gamma (u_1+\varphi_1, u_2+ \varphi_2)^{T} \mathcal{A} \left( \begin{array}{c}
u_1+ \varphi_1\\
u_2+ \varphi_2
\end{array} \right) - (u_1, u_2)^T \mathcal{A}  \left( \begin{array}{c}
u_1\\
u_2
\end{array} \right)   \ud x
\end{aligned}
\end{equation*}
and the perturbed total energy on $\Gamma$ as
\begin{equation*}
\heg:= \hegg+\int_\Gamma W(u_1+\varphi_1, u_2+ \varphi_2)-W(u_1, u_2) \ud x.
\end{equation*}

One can check the straight solution uniform in $x_2$, i.e., $u_1(x_1,x_2)= \phi(x_1), \, u_2(x_1,x_2)=0$ is a solution to \eqref{eq2D}, where $ \phi(x_1)$ is the solution to the 1D problem
\begin{equation}\label{1d-old}
\begin{aligned}
(-\Delta)^{\frac12}\phi(x_1) &= -(1-\nu) W'(\phi(x_1)), \quad x_1\in\mathbb{R}\\
\lim_{x_1\to \pm\8}\phi(x_1)&=\pm 1.
\end{aligned}
\end{equation}
We refer to \cite{CS05, PSV13} for {the} existence and uniqueness to \eqref{1d-old}, which also proved {that} $\phi$ is bounded, increasing from $-1$ to $1$, and a local minimizer of the corresponding 1D energy. See also \cite{GL19} in which the authors proved that $\phi$ is the unique equilibrium of the corresponding 1D nonlocal dynamics Ginzburg-Landau equation.
However, {there might be other solutions} to \eqref{maineq}. In {the} next section, we will further reduce the nonlocal system to a 1D nonlocal equation for the case potential depending only on $[u_1]$ and {characterize} bounded stable solutions.

\subsubsection{Reduction of the 2D nonlocal system  to a 1D equation}
If the misfit potential $W$ depends only on one component of displacement jump, i.e., $W(u_1, u_2)=W(u_1)$, we can reduce the 2D system further to a scalar equation. Let us first clarify the assumption on the double well/periodic potential $W$:
\begin{equation}\label{potential}
\begin{aligned}
 &W \in C_b^{2, \alpha}(\mathbb{R}; \bR), \vspace{1ex} \\ &W(x)>W(\pm 1),\quad x \in \left(-1, 1 \right), \vspace{1ex} \\
 &  W''\left(\pm 1\right)>0.
 \end{aligned}
 \end{equation}
In the case {when} $W(u_1, u_2)=W(u_1)$, from \eqref{curve}, we
 represent $\hat{u}_2$ by $\hat{u}_1$, i.e.,
\begin{equation}\label{u13}
 \dfrac{\nu}{1-\nu}\dfrac{k_1k_2}{|k|}\hat{u}_1(k)+\left(\dfrac{k_1^2}{|k|}+\dfrac{1}{1-\nu}\dfrac{k_2^2}{|k|}\right)\hat{u}_2(k)=0,
\end{equation}
  which is equivalent to
  $$
  \hat{u}_2(k)=-\dfrac{\nu k_1k_2}{(1-\nu)k_1^2+k_2^2}\hat{u}_1(k).
  $$
Substituting this equality {in} the first component in \eqref{curve} yields
\begin{align}
    \hat{\sigma}_{13}(k)&=-\left[\left(\dfrac{k_2^2}{|k|}+\dfrac{1}{1-\nu}\dfrac{k_1^2}{|k|}\right)\hat{u}_1(k)+\dfrac{\nu}{1-\nu}\dfrac{k_1k_2}{|k|}\hat{u}_2(k)\right]\nonumber\\
    &=-\dfrac{|k|^3}{(1-\nu)k_1^2+k_2^2}\hat{u}_1(k)=\F( W'(u_1)).\label{original}
\end{align}
Therefore, the 2D system \eqref{eq2D} is reduced to a new 1D nonlocal equation
\begin{equation}\label{maineq1D}
\begin{aligned}
\mathcal{L} u_1(x_1, x_2) = -W'(u_1(x_1,x_2)),& \quad (x_1,x_2)\in\Gamma,\\
 \lim_{x_1\to\pm \8}u_1(x_1,x_2)=\pm 1,&\quad x_2\in \mathbb{R},
 \end{aligned}
\end{equation}
where the nonlocal operator $\mathcal{L}$ has the Fourier symbol
\begin{equation*}
\frac{|k|^3}{(1-\nu)k_1^2 + k_2^2}\in \big[\frac{|k|}{2}~, ~2 |k|\big],~ \quad \nu\in[-1, \frac12].
\end{equation*}
Later in Section \ref{sec3}, we will derive the integral formulation of the nonlocal operator $\cL$ and study the properties of its kernel.
Compared {to} \eqref{energy2d-o}, we  also have the corresponding (further) reduced energy $E_\Gamma^0$ on $\Gamma$
\begin{equation}\label{energy2d-n}
E_\Gamma^0:= \frac 1 2\int_{\bR^2} u_1 \cL u_1 \ud x + \int_{\bR^2} W(u_1) \ud x,
\end{equation}
which is equivalent to $E_\Gamma$ in \eqref{energy2d-o} in {the} perturbed or local sense; see detailed arguments for the perturbed sense in \cite{GLLX19}.
Notice that the nonlinearity is now {coupled} to only $u_1$ on $\Gamma$. The main goal is to study the existence, uniqueness, and the property of solution to \eqref{maineq1D}. If one can solve \eqref{maineq1D}, then by elastic extension introduced in \cite{GLLX19}, we obtain the vector-field solutions to the original 3D full system \eqref{maineq}.

Recall that the straight solution (uniform in $x_2$), i.e., $u_1(x_1,x_2)= \phi(x_1), \, u_2(x_1,x_2)=0$ is a also solution to \eqref{maineq1D}, where $ \phi(x_1)$ is the solution to the 1D problem \eqref{1d-old}.
For notation simplicity, from now on, we replace $u_1(x_1,x_2)$ with a scalar function $u(x): \bR^2 \to \bR$ in \eqref{maineq1D} and recast \eqref{maineq1D} to
\begin{equation}\label{1D-neq}
\begin{aligned}
\mathcal{L} u(x) = -W'(u(x)), \quad &x=(x_1,x_2)\in \bR^2,\\
 \lim_{x_1\to\pm \8}u(x_1,x_2)=\pm 1,\quad &x_2\in \mathbb{R}.
 \end{aligned}
\end{equation}
We will focus on the kernel representation of the operator $\cL$ in Section \ref{sec3} and then prove {that} the solution $u(x)$ to \eqref{1D-neq} must {have} a 1D profile in Section \ref{1Dprop}. As a consequence, we will {finally} prove {that} the straight dislocation is the {only} stable solution (up to a rotation and translations)  to the full system \eqref{maineq} in Theorem \ref{thm-1d},

\section{Positive and anisotropic kernel of $\mathcal{L}$}\label{sec3}
In this section, we derive the integral formulation  of the operator $\cL$ in the Schwartz   space $\cS(\bR^2)$  and prove certain properties of its singular kernel.  We will use this integral formulation for $\mathcal{L}$ whenever the singular integration make sense, for instance, on {the} space $\{u\in \dot{H}^s(\bR^2) \, \text{ for any } s\geq 1\}$. In the remaining part of this paper, $C$ is a generic  constant whose value may change from line to line.

Recall the integral formulation of the half Laplacian $\Lambda:=(-\Delta)^{\frac12}$ on $ \cS(\bR^2)$
\begin{equation*}
\Lambda u  = -\frac{C_d}{2} \int_{\bR^2} \big( u(x+y)+ u(x-y)-2u(x)\big)|y|^{-3} \ud y,
\end{equation*}
where $$
C_d:=\frac{2}{\pi}\frac{\Gamma(\frac32)}{|\Gamma(-\frac12)|}=\frac{1}{2\pi}.
$$
First we state a lemma for the solution to an elliptic equation, whose proof will be given later.
\begin{lem}\label{lem_kernel}
Let $\beta:=1-\nu\in [\frac12,2]$. The elliptic equation
\begin{equation}\label{ellip}
\Delta P(x_1,x_2) = \frac{1}{(\beta x_1^2 + x_2^2)^{\frac52}}, \quad (x_1,x_2)\in\bR^2\backslash\{0\}
\end{equation}
 has a solution $$
 P(x_1,x_2)= \frac{v(\theta)}{(x_1^2+ x_2^2)^{\frac32}},
 $$ where $\theta=\arctan\frac{x_2}{x_1}$ and $v(\theta)$ is the unique $\pi$-periodic solution to
\begin{equation}\label{vODE}
v''+9v=(\beta \cos^2 \theta + \sin^2 \theta)^{-\frac{5}{2}}.
\end{equation}
Moreover, we have the following properties of $v(\theta)$
\begin{enumerate}[(i)]
\item $v(\theta)$ is symmetric with respect to $\frac{\pi}{2}$;
\item For $\beta\geq 1$, $v(\theta)$ is increasing in $[0,\frac{\pi}{2}]$ and decreasing in $[\frac{\pi}{2}, \pi]$; while for $0< \beta< 1$, $v(\theta)$ is decreasing in $[0,\frac{\pi}{2}]$ and increasing in $[\frac{\pi}{2}, \pi]$;
\item For $\frac23<\beta<\frac32$, $v(\theta)$ is positive and $\frac{1}{9}c_\beta \leq v(\theta)\leq \frac19$ for $0\leq \theta \leq \pi$, where
    $$c_\beta:=\min\{\frac{3\beta-2}{\beta^2}, \frac{3-2\beta}{\beta^{\frac32}}\}>0.$$
\end{enumerate}
\end{lem}

In  Proposition \ref{prop_kernel1}, we derive the corresponding integral formulation for $\cL$ and then study the properties of the singular kernel in Proposition \ref{prop_kernel2}.
\begin{prop}\label{prop_kernel1}
The integral formulation of $\cL$ is given by
\begin{equation*}
\cL u=-\frac{1}{4\pi}\int_{\bR^2} \big(u(x+y)+u(x-y)-2u(x)\big)K(y)\,\ud y,
\end{equation*}
where $K(y):=9P(y_1/\sqrt\beta,y_2)$ satisfies
\begin{equation}
                        \label{eq3.13}
(\beta \pt_1^2 + \pt_2^2 )K(y) = \frac{9}{|y|^5},\quad\forall y\in\bR^2\backslash\{0\}.
\end{equation}
\end{prop}
\begin{proof}
Step 1. We {first} derive the integral formulation of $\Lambda^3$, where $\Lambda=(-\Delta)^{1/2}$. For any $u$ in the Schwartz class $\cS(\bR^2)$, we have
\begin{align*}
\Lambda^3 u(x)
&=\frac{1}{4\pi}\int_{\bR^2}\big(\Delta_x u(x+y)+\Delta_x u(x-y)-2\Delta_x u(x)\big)|y|^{-3}\,\ud y\\
&=\frac{1}{4\pi}\int_{\bR^2}\Delta_y\Big( u(x+y)+ u(x-y)-2u(x)-\sum_{i=1,2}y_i^2(\pt_i^2 u)(x)\Big)|y|^{-3}\,\ud y\\
&=\lim_{\varepsilon\to 0}\frac{1}{4\pi}\int_{B_\varepsilon^c}\Delta_y\Big( u(x+y)+ u(x-y)-2u(x)-\sum_{i=1,2}y_i^2(\pt_i^2 u)(x)\Big)|y|^{-3}\,\ud y\\
&=\lim_{\varepsilon\to 0}\left[\frac{9}{4\pi}\int_{B_\varepsilon^c}\Big( u(x+y)+ u(x-y)-2u(x)-\sum_{i=1,2}y_i^2(\pt_i^2 u)(x)\Big)|y|^{-5}\,\ud y+I_1\right],
\end{align*}
where we applied Green's identity in the last equality and $I_1$ is the boundary term. Since
$$
|u(x+y)+ u(x-y)-2u(x)-\sum_{i=1,2}y_i^2(\pt_i^2 u)(x)|\le c|y|^4,
$$
we have $I_1\sim O(\varepsilon)\to 0$ as $\varepsilon\to 0$. Therefore, we obtain
$$
\Lambda^3 u(x)=\frac{9}{4\pi}\int_{\bR^2}\big( u(x+y)+ u(x-y)-2u(x)-\sum_{i=1,2}y_i^2(\pt_i^2 u)(x)\big)|y|^{-5}\,\ud y.
$$

Step 2. We show that for any $u\in \cS$,
$$
\cF^{-1}\Big(\frac{|k|^3}{\beta k_1^2+k_2^2}\hat u\Big)
=-\frac{1}{4\pi}\int_{\bR^2} \big(u(x+y)+u(x-y)-2u(x)\big)K(y)\,dy,
$$
where $\beta=1-\nu$ and $K(y)$ satisfies \eqref{eq3.13}.

Recall that $P(y)$ is the solution to \eqref{ellip} we obtained in Lemma \ref{lem_kernel}, and thus $K(y)\sim |y|^{-3}$ is homogeneous of degree $-3$.
  By using a cutoff near the origin and the dominated convergence theorem, we may assume that $\hat u$ vanishes near the origin. Let $u=Lv$, where $v$ is also in $\cS$ and the second-order operator $L:=-\beta \pt_1^2-\pt_2^2$ has the symbol $\beta k_1^2+k_2^2$. In other words,
$$
\hat v=\frac{\hat u}{\beta k_1^2+k_2^2}.
$$
It  suffices to show that
\begin{equation}\label{ani_k}
\cF^{-1}(|k|^3 \hat{v})=\Lambda^3 v(x)=-\frac{1}{4\pi}\int_{\bR^2} \big(L_xv(x+y)+L_xv(x-y)-2L_xv(x)\big)K(y)\,\ud y.
\end{equation}
By using a similar computation, the right-hand side above is equal to
\begin{equation}\label{rhs_k}
\begin{aligned}
&-\frac{1}{4\pi}\int_{\bR^2} L_y\big(v(x+y)+v(x-y)-2v(x)-\sum_{i=1,2}y_i^2(\pt_i^2 v)(x)\big)K(y)\,\ud y\\
&=\lim_{\varepsilon\to 0}-\frac{1}{4\pi}\int_{B_\varepsilon^c} L_y\big(v(x+y)+v(x-y)-2v(x)-\sum_{i=1,2}y_i^2(\pt_i^2 v)(x)\big)K(y)\,\ud y\\
&=\lim_{\varepsilon\to 0}-\frac{1}{4\pi}\left[\int_{B_\varepsilon^c} \big(v(x+y)+v(x-y)-2v(x)-\sum_{i=1,2}y_i^2(\pt_i^2 v)(x)\big)L_y K(y)\,dy+I_2\right],
\end{aligned}
\end{equation}
where we applied Green's identity in the last equality and $I_2$ is the boundary term. As before, $I_2\sim O(\varepsilon)\to 0$ as $\varepsilon\to 0$. Because $L_y K(y)=-9|y|^{-5}$ for $y\neq 0$, the last limit in \eqref{rhs_k} is equal to
$$
\frac{9}{4\pi}\int_{\bR^2} \big(v(x+y)+v(x-y)-2v(x)-\sum_{i=1,2}y_i^2(\pt_i^2 v)(x)\big)|y|^{-5}\,\ud y=\Lambda^3 v(x),
$$
which yields \eqref{ani_k}.
\end{proof}

Combining Lemma \ref{lem_kernel} and Proposition \ref{prop_kernel1}, we obtain an  anisotropic kernel $K$.
Since $P(y)=\frac{1}{9}K(\sqrt{\beta}y_1, y_2)$ solves \eqref{ellip}, by {a change} of variables
\begin{equation*}
(\bar{x}_1, \bar{x}_2) = (\frac{1}{\sqrt{\beta}} x_1, x_2), \quad \bar{u}(\bar{x}_1, \bar{x}_2):= u(\sqrt{\beta}\bar{x}_1, \bar{x}_2),
\end{equation*}
we know {that} if $u(x_1, x_2)$ is {a} solution to \eqref{1D-neq}, then $\bar{u}(\bar{x}_1, \bar{x}_2)$ is a solution to
\begin{equation}\label{1D-maineq}
\begin{aligned}
\bar{\cL} \bar{u} = -\frac{1}{\sqrt{\beta}}W'(\bar{u}), \quad &\bar{x}\in \bR^2,\\
 \lim_{\bar{x}_1\to\pm \8}\bar{u}(\bar{x}_1, \bar{x}_2)=\pm 1,\quad &\bar{x}_2\in \mathbb{R},
 \end{aligned}
\end{equation}
where the nonlocal operator $\bar{\cL}$ is given by
\begin{equation}\label{bar-opt}
\begin{aligned}
\bar{\cL} \bar{ u}=-\frac{1}{4\pi}\int_{\bR^2} \big(\bar{u}(\bar{x}+\bar{y})+\bar{u}(\bar{x}-\bar{y})-2\bar{u}(\bar{x})\big)\bar{K}(\bar{y})\,\ud \bar{y}, \quad \bar{K}(\bar{y}):=\frac{9v(\theta)}{|\bar{y}|^{3}},
\end{aligned}
\end{equation}
with $v(\theta)=v(\arctan\frac{\bar{y}_2}{\bar{y}_1})$ being the solution to \eqref{vODE}.  In {the} next section, we will focus on the analysis of the solution to \eqref{1D-maineq} and drop the bar in \eqref{1D-maineq}.
 Now we summarize the properties  of the kernel $\bar{K}$ below.
\begin{prop}\label{prop_kernel2}
For $\frac23<\beta<\frac32$, the kernel $\bar{K}$ of $\bar{\cL}$ in \eqref{bar-opt} is positive and satisfies the following
properties
\begin{enumerate}[(i)]
\item $\bar{K}(-x)=\bar{K}(x), \quad \bar{K}(a x) = a^{-3} \bar{K}(x) \text{ for any }a>0$;
\item $0<\frac{c_\beta}{|x|^3}\leq \bar{K}(x)\leq \frac{1}{|x|^3};$
\item $ \max\{|x||\pt_e \bar{K}|, |x|^2|\pt_{ee}\bar{K}|\}\leq \frac{C}{|x|^3}$
\end{enumerate}
for {any} $x\in\bR^2\backslash\{0\}$ and {unit vector} $e\in S^{1}$, where $c_\beta$ {is} defined in Lemma \ref{lem_kernel} (iii).
\end{prop}

\begin{cor} (Strict positivity property at global minima and global maxima)\label{pp}
For any function $g(\mathbf{w})\in C(\bR^2)$, let $\mathbf{w}_m=(x_m,y_m),\,  \mathbf{w}_M=(x_M, y_M) \in \bR^2$ be the points at which $g(\mathbf{w})$ attains it global minimum and maximum respectively. Then we have
\begin{equation*}
\bar{\ptf} g(\mathbf{w})|_{\mathbf{w}=\mathbf{w}_m}<0, \quad \bar{\ptf} g(\mathbf{w})|_{\mathbf{w}=\mathbf{w}_M}>0
\end{equation*}
provided $g(\mathbf{w})$ is not a constant.
\end{cor}
\begin{proof}
From the positivity of the kernel $\bar{K}$ in Proposition \ref{prop_kernel2}, since $g(\mathbf{w}_m)\leq g(\mathbf{w})$ for all $\mathbf{w}\in\mathbb{R}\times \mathbb{T}$,  we have
$$\bar{\ptf} g(\mathbf{w})|_{\mathbf{w}=\mathbf{w}_m} \leq 0$$
and the equality holds if and only if $g(\mathbf{w})\equiv g(\mathbf{w}_m)$ for all $\mathbf{w}\in\mathbb{R}\times \mathbb{T}$. The proof for $\bar{\ptf}g$ at $\mathbf{w}_M$ is {the} same.
\end{proof}

We finish this section by proving Lemma \ref{lem_kernel}.
\begin{proof}[Proof of Lemma \ref{lem_kernel}]
Step 1. To solve
\begin{equation}\label{laplace2}
\Delta {P}(x_1,x_2) = \frac{1}{(\beta x_1^2+ x_2^2)^{\frac52}} , \quad (x_1,x_2)\in\bR^2\backslash\{0\},
\end{equation}
by a change of variables
\begin{equation}\label{variables}
x_1=r \cos \theta,\,\, x_2= r \sin \theta, \quad {P}(x,y)= r^{-3} v(\theta)
\end{equation}
in \eqref{laplace2}, we have the ODE for $v(\theta)$ \eqref{vODE}, i.e.,
\begin{equation*}
v''+ 9 v = (\beta \cos^2 \theta + \sin^2 \theta)^{-\frac52},
\end{equation*}
where $\beta\in[\frac12,2]$.
Notice  that the natural period for the harmonic oscillation $v''+9v=0$ is $\frac{2\pi}{3}$ while the force term
\begin{align*}
f(\theta):=(\beta \cos^2 \theta + \sin^2 \theta)^{-\frac52}= [\frac{\beta+1}{2}+\frac{\beta-1}{2}\cos(2\theta)]^{-\frac52}
\end{align*}
has period $\pi$.
 Therefore, we always has a $2\pi$-periodic solution to the ODE   \eqref{vODE}. Besides, from elementary calculations, one can check that for $\beta\leq 1$,
 $$f_{\min}=f(\frac{\pi}{2}+k\pi), \quad f_{\max}=f(k\pi), \quad k\in\mathbb{Z},$$
while for $\beta\geq 1$,
 $$f_{\max}=f(\frac{\pi}{2}+k\pi), \quad f_{\min}=f(k\pi), \quad k\in\mathbb{Z}.$$

Step 2. Existence and uniqueness of a $\pi$-periodic solution.

First, we know {that} $P$ satisfies $P(-x, -y)=P(x,y)$, which, together with \eqref{variables}, yields the periodicity $v(\theta+\pi)=v(\theta)$. Therefore, we seek {a} periodic solution to \eqref{vODE} with period $\pi$.

Second, by {the} method of variation of parameters, one can solve a special solution $v_0(\theta)$
\begin{equation}\label{special}
\begin{aligned}
v_0(\theta)=&u_1(\theta)\cos(3\theta)+u_2(\theta)\sin(3\theta) \\&\text{ with } u_1(\theta)=-\frac13\int_0^{\theta} \sin (3 x) f(x) \ud x,\, u_2(\theta)=\frac13 \int_0^{\theta} \cos (3 x)f(x) \ud x
\end{aligned}
\end{equation}
 and thus the general solution  to \eqref{vODE} is given by
\begin{equation}\label{v_sol}
v(\theta)= C_1 \cos(3\theta)+ C_2 \sin(3\theta) + v_0(\theta).
\end{equation}
Notice that for any $\pi$-periodic function $v(\theta)$, we have
\begin{equation}\label{pi}
\int_{-\pi}^{\pi} v(\theta) \cos(k\theta) \ud \theta = 0,\quad \int_{-\pi}^{\pi} v(\theta) \sin(k\theta) \ud \theta = 0 \quad \text{ for any odd integer }k.
\end{equation}
Therefore, to obtain a $\pi$-periodic solution,
we must set
\begin{equation}\label{c12}
C_1 := -\frac{1}{\pi}\int_{-\pi}^{\pi} v_0(\theta) \cos(3\theta) \ud \theta, \quad C_2 :=- \frac{1}{\pi}\int_{-\pi}^{\pi} v_0(\theta) \sin(3\theta) \ud \theta.
\end{equation}

Third, we check $v(0)=v(\pi)$ and $v'(0)=v'(\pi)$.

By plugging in, we have
\begin{equation*}
v(0)=v_0(0)+C_1= -\frac{1}{\pi}\int_{-\pi}^{\pi} v_0(\theta) \cos(3\theta) \ud \theta
\end{equation*}
and
\begin{equation*}
v(\pi)=v_0(\pi)-C_1=\frac13\int_0^{\pi} \sin(3x) f(x) \ud x +\frac{1}{\pi}\int_{-\pi}^{\pi} v_0(\theta) \cos(3\theta) \ud \theta.
\end{equation*}
From \eqref{special} we know {that} $u_1(\theta)$, $u_2(\theta)$, and thus $v_0(\theta)$ are all periodic {functions} with period $2\pi$. Hence by integration by parts, we have
\begin{align*}
\int_{-\pi}^{\pi} v_0(\theta) \cos(3\theta) \ud \theta =& -\frac13 \int_{-\pi}^\pi v_0'(\theta) \sin 3\theta \ud \theta\\
=& -\frac13 \int_{-\pi}^\pi [u_1(\theta)(-3\sin 3\theta )+u_2(\theta)(3\cos 3\theta)]\sin(3\theta) \ud \theta\\
=&\int_{-\pi}^\pi u_1(\theta) \sin^2(3\theta) - u_2(\theta) \cos 3\theta \sin 3\theta \ud \theta.
\end{align*}
Since $f(x)$ has period $\pi$ and \eqref{pi}, one can check
\begin{equation}\label{u1_46}
\begin{aligned}
\int_{-\pi}^\pi u_1(\theta) \sin^2(3\theta) \ud \theta =&-\frac13 (\frac{\theta}{2}-\frac{\sin 6\theta}{12}) \int_0^\theta (\sin 3x)f(x) \ud x \Big|_{-\pi}^\pi+\frac16 \int_{-\pi}^{\pi}\theta \sin 3\theta f(\theta) \ud \theta\\
=&-\frac{\pi}{6} \int_0^\pi (\sin 3\theta) f(\theta) \ud \theta,
\end{aligned}
\end{equation}
where we used
$$
\int_{-\pi}^\pi \theta \sin 3\theta f(\theta) \ud \theta = \pi \int_0^\pi \sin 3\theta f(\theta) \ud \theta.
$$
Similarly, we obtain
\begin{equation*}
\int_{\pi}^\pi u_2(\theta) \cos 3\theta \sin 3\theta \ud \theta= 0.
\end{equation*}
Therefore, we conclude that
\begin{equation}\label{cc1}
\frac{1}{\pi}\int_{-\pi}^{\pi} v_0(\theta) \cos(3\theta) \ud \theta= -\frac16\int_0^\pi (\sin 3\theta)f(\theta) \ud \theta,
\end{equation}
which yields
\begin{equation*}
v(0)= \frac16\int_0^\pi (\sin 3x)f(x) \ud x = v(\pi).
\end{equation*}

Then by plugging in, we have
\begin{equation*}
v'(0)=v_0'(0)+3C_2= -\frac{3}{\pi}\int_{-\pi}^{\pi} v_0(\theta) \sin(3\theta) \ud \theta
\end{equation*}
and
\begin{equation*}
v'(\pi)=v_0'(\pi)-3C_2=-\int_0^{\pi} \cos(3\theta) f(\theta) \ud \theta +\frac{3}{\pi}\int_{-\pi}^{\pi} v_0(\theta) \sin(3\theta) \ud \theta.
\end{equation*}
By the similar calculation in \eqref{u1_46}, we have
\begin{equation}\label{cc2}
\int_{-\pi}^{\pi} v_0(\theta) \sin(3\theta) \ud \theta= \int_{-\pi}^{\pi} u_2(\theta) \cos^2(3\theta) \ud \theta = \frac{\pi}{6} \int_0^\pi \cos(3\theta)f(\theta) \ud \theta.
\end{equation}
Therefore, we verified $v'(0)=v'(\pi)$. Thus from the uniqueness of the solution to ODE \eqref{vODE} we conclude that \eqref{v_sol} with coefficients in \eqref{c12} is the unique $\pi$-periodic solution to \eqref{vODE}. From \eqref{cc1} and \eqref{cc2}, we have
\begin{equation*}
\begin{aligned}
v(\theta)=& \frac16 \left( 2\int_0^\theta \cos 3x f(x) \ud x - \int_0^\pi \cos 3x f(x) \ud x \right) \sin 3\theta\\
&+ \frac16 \left( -2\int_0^\theta \sin 3x f(x) \ud x + \int_0^\pi \sin 3x f(x) \ud x \right) \cos 3\theta\\
=& \frac16\left( \int_0^\theta \cos 3x f(x) \ud x - \int_\theta^\pi \cos 3x f(x) \ud x  \right) \sin 3\theta\\
&-  \frac16\left( \int_0^\theta \sin 3x f(x) \ud x - \int_\theta^\pi \sin 3x f(x) \ud x  \right) \cos 3\theta
\end{aligned}
\end{equation*}
and
\begin{equation*}
\begin{aligned}
v'(\theta)=&\frac{\cos 3\theta}{2} \left( \int_0^\theta \cos 3x f(x) \ud x -\int_\theta^\pi \cos 3x f(x) \ud x \right)\\
 &+ \frac{\sin 3\theta}{2} \left( \int_0^\theta \sin 3x f(x) \ud x -\int_\theta^\pi \sin 3x f(x) \ud x \right)
\end{aligned}
\end{equation*}
for $0\leq \theta \leq \pi.$

Step 3. Properties of $v(\theta)$ and the range of {$\beta$} such that $v$ is positive.

Denote $g_1(x):=(\cos 3 x) f(x)$ and $g_2(x):=(\sin 3 x) f(x)$, which have the symmetric property
\begin{equation*}
g_1(\frac{\pi}{2}+ x) = - g_1(\frac{\pi}{2}-x), \quad g_2(\frac{\pi}{2}+ x)=g_2(\frac{\pi}{2}-x).
\end{equation*}
Therefore, for $0\leq \theta<\pi$, we have
\begin{equation*}
\left( \int_{0}^\theta - \int_{\theta}^\pi \right) g_1(x) \ud x = 2 \int_0^\theta g_1(x) \ud x, \quad \left( \int_{0}^\theta - \int_{\theta}^\pi \right) g_2(x) \ud x = -2\int_{\theta}^{\frac{\pi}{2}} g_2(x) \ud x
\end{equation*}
and thus $v(\theta)$ and $v'(\theta)$ can be expressed as
\begin{align*}
v(\theta)= \frac{\sin 3 \theta}{3} \int_0^\theta g_1(x) \ud x + \frac{\cos 3 \theta}{3} \int_{\theta}^{\frac{\pi}{2}} g_2(x) \ud x, \quad 0\leq \theta < \pi,\\
v'(\theta)=\cos 3\theta \int_0^{\theta} g_1(x) \ud x - {\sin 3\theta}\int_{\theta}^{\frac{\pi}{2}} g_2(x) \ud x,\quad 0\leq \theta < \pi.
\end{align*}
Moreover, we have
\begin{equation}\label{sym-d1}
v(\theta)=v(\pi-\theta),\quad v'(\theta)=-v'(\pi-\theta).
\end{equation}
Now we give the following claim:\\
\textit{ For $0\leq \theta<\pi$, the equation $v'(\theta)=0$ only has two roots $\theta=0,\pi$.}

\begin{proof}
Indeed, we only need to prove this claim for $\beta>1$. For the case $0<\beta<1$, denote $\bar{\beta}:=\frac{1}{\beta}$, then by $v'(\frac{\pi}{2}-\theta,\beta)=\bar{\beta}^{\frac52}v'(\theta,\bar{\beta})$, the problem is reduced to the case $\beta>1.$

Denote $w(\theta):=v'(\theta)$. Then $w$ satisfies $w''+9w=f'(\theta)$.
For $\beta>1$, we know {that} $f'(\theta)>0$ in $(0,\pi/2)$. By the symmetric property for $v'(\theta)$ in \eqref{sym-d1}, it remains to prove that the solution to
\begin{equation}\label{BVP}
w''+9w=f'(\theta)>0, \quad w(0)=w(\frac{\pi}{2})
\end{equation}
is strictly positive for $0<\theta<\frac{\pi}{2}$.  If it is not true, then there exist $0<a\leq b<\frac{\pi}{2}$ such that $w(a)=w(b)=0$ and $w(\theta)>0$ for $\theta\in(0,a)\cup (b,\frac{\pi}{2})$. Notice that the eigenvalue problem
\begin{equation*}
w''+\lambda w=0, \quad w(0)=w(a)=0
\end{equation*}
has the smallest eigenvalue $\lambda_1=\left( \frac{\pi}{a} \right)^2$. If $a\leq \frac{\pi}{3}$, then $\lambda_1\geq 9$. However, this is impossible because $w(\theta)>0$ for $\theta\in(0,a)$ and
\eqref{BVP} implies
\begin{equation*}
\int_{0}^a w(w''+\lambda_1 w) \ud \theta\ge
\int_{0}^a w(w''+9 w) \ud \theta = \int_0^a f'(\theta) w \ud \theta>0.
\end{equation*}
Thus we conclude that $a>\frac{\pi}{3}$. Similarly, since the eigenvalue problem
\begin{equation*}
w''+\lambda w=0, \quad w(b)=w(\frac{\pi}{2})=0
\end{equation*}
has the smallest eigenvalue $\lambda_1=\left( \frac{\pi}{\frac{\pi}{2}-b} \right)^2 $, we conclude that $b<\frac{\pi}{6}$. This gives a contradiction and we complete the proof.
\end{proof}

Then by elementary calculations, we have
\begin{equation*}
\begin{aligned}
v(0) =& \frac{1}{3} \int_0^{\frac{\pi}{2}} \frac{\sin 3x}{(\beta \cos^2 x + \sin^2 x)^{\frac52}} \ud x= -\frac13 \int_0^1 \frac{1-4t^2}{(1+(\beta-1)t^2)^{\frac52}}\ud t\\
\overset{t=\frac{\tan y}{\sqrt{\beta-1}}}{=}& -\frac13 \int_0^{\arctan \sqrt{\beta-1}} \frac{1}{\sqrt{\beta-1}} (1-\frac{4}{\beta-1} \tan ^2 y) \frac{1}{\sec^3 y}\ud y\\
\overset{s=\sin y}{=}& -\frac13 \int_0^{\frac{\sqrt{\beta-1}}{\sqrt{\beta}}} \frac{1}{\sqrt{\beta-1}}(1-\frac{\beta+3}{\beta-1} s^2) \ud s
= \frac{3-2\beta}{9\beta^{\frac32}}.
\end{aligned}
\end{equation*}
Similarly, we have
\begin{equation*}
v(\frac{\pi}{2}) = -\frac{1}{3} \int_0^{\frac{\pi}{2}} \frac{\cos 3x}{(\beta \cos^2 x + \sin^2 x)^{\frac52}} \ud x= -\frac13 \int_0^1 \frac{1-4t^2}{(\beta+(1-\beta)t^2)^{\frac52}}\ud t= \frac{3\beta-2}{9\beta^2}.
\end{equation*}

On one hand, if $\frac12\leq \beta\leq 1$,
$$  \frac{3\beta-2}{9\beta^2}=v(\frac{\pi}{2})\leq v(\theta)\leq v(0)=\frac{3-2\beta}{9\beta^{\frac32}} \quad \text{ for any } 0\leq \theta\leq \pi.$$
In this case, $v(\frac{\pi}{2})=0$ if and only if $\beta=\frac23$ and thus
\begin{equation*}
v_{\min}= \frac{3\beta-2}{9\beta^2}>0, \quad v_{\max}=\frac{3-2\beta}{9\beta^{\frac32}}\leq \frac{1}{9} \quad \text{ for }\frac23<\beta\leq 1.
\end{equation*}
On the other hand, if $1\leq \beta\leq 2$,
$$  \frac{3\beta-2}{9\beta^2}=v(\frac{\pi}{2})\geq v(\theta)\geq v(0)=\frac{3-2\beta}{9\beta^{\frac32}} \quad \text{ for any } 0\leq \theta\leq \pi.$$
In this case, $v(0)=0$ if and only if $\beta=\frac32$
and thus
\begin{equation*}
v_{\min}= \frac{3-2\beta}{9\beta^{\frac32}}>0, \quad v_{\max}=\frac{3\beta-2}{9\beta^2}\leq \frac{1}{9} \quad \text{ for }1\leq\beta<\frac32.
\end{equation*}
Therefore, we conclude that when $\frac23<\beta<\frac32$, there exists a unique $\pi$-periodic positive solution $v(\theta)$ to \eqref{vODE}.
\end{proof}

\section{Bounded stable solution has a 1D profile}\label{1Dprop}

In this section, we will prove {that any} bounded stable solution to \eqref{1D-maineq}, dropping bars for notation simplicity, {has} a 1D profile, i.e., $u(x)=\phi(e\cdot x)$ for some $e\in S^1$, where $\phi$  is the unique (up to translations)  solution to a 1D problem; see Theorem \ref{thm-1d}. From  \cite{CS05, PSV13}, {we know that} $\phi$ is bounded, increasing from $-1$ to $1$, and a local minimizer of the corresponding energy. The proof relies on the local BV estimates originally developed by \cite{CSV19} to study the {quantitative}  flatness of nonlocal minimal surface. Their method does not use any extension {argument} and thus is particularly powerful for the nonlocal problem with general anisotropic kernel. This is the key in our case {as} we do not have a scalar-valued extended 3D problem. In this section, we will always assume $\frac23<\beta<\frac32$ so that we have good properties of the kernel $\bar{K}$ in Proposition \ref{prop_kernel2}.

The proof of the 1D profile is divided into the following three subsections, which roughly say that stability implies flatness; c.f. \cite{dipierro2016nonlocal, Gui19, FS20, CSV19, savin2018rigidity}.
First, let us clarify the definition of stable solutions and how to define the  perturbations to {these} stable solutions in a ball $B_R$ with respect to some direction $\cnu$.
Define the total energy of $u$ in any ball $B_{R}\subset\bR^2$ as
\begin{equation}\label{B_energy}
\begin{aligned}
E_\Gamma^0(u; B_R):
&=\frac{C_d}{4}\iint_{\bR^2\times \bR^2\backslash B_R^c\times B_R^c} |u(x)-u(y)|^2\bar{K}(x-y) \ud x \ud y + \frac{1}{\sqrt{\beta}}\int_{B_R} W(u(x)) \ud x\\
&=:\frac{C_d}{4}\cE(u;B_R)+ F(u;B_R),
\end{aligned}
\end{equation}
where $\bar{K}$ is the kernel in \eqref{bar-opt} satisfying {the} properties in Proposition \ref{prop_kernel2}.
Here the nonlocal energy $\cE(u;B_R)$ can be viewed as the contribution in $B_R$ of the semi-norm $\|\cdot\|_{\dot{H}^\frac12(\bR^2)}$ because we formally have
$$
\|u\|_{\dot{H}^\frac12}^2=\lim_{R\to +\8} \cE(u;B_R).
$$
\begin{defn}\label{def_stable}
We say that $u$ is a stable solution to \eqref{1D-maineq} if the second  local variation  of  $E_\Gamma^0$ {defined in \eqref{energy2d-n}} is nonnegative, i.e.,
\begin{equation*}
\int_{\bR^2} \left(\bar{\cL} v + \frac{1}{\sqrt{\beta}}W''(u)v \right) v \ud x \geq 0 \quad \text{for any }v\in C_c^2(\bR^2).
\end{equation*}
\end{defn}
Next, following \cite{CSV19} we define the perturbations to {these} stable solutions in a ball $B_R$ with respect to some direction $\cnu$. Let $R\geq 1$, define the  perturbed coordinates along $\cnu\in S^1$ direction as
\begin{equation*}
\psi_{t,\cnu}(z):= z+ t \varphi(z) \cnu,
\end{equation*}
where $\varphi$ is a cut-off function compact supported in $B_R$
\begin{equation*}
\varphi(z)= \left\{\begin{array}{cc}
1, \quad & |z|\leq \frac{R}{2},\\
2-\frac{2|z|}{R},\quad & \frac{R}{2} \leq |z|\leq R,\\
0, \quad  &|z|\geq R.
\end{array}\right.
\end{equation*}
Since for $t$ small enough, $\psi_{t, \cnu}$ is invertible, the local perturbed solution is defined by the pushforward operator
\begin{equation*}
P_{t, \cnu}u(x)= u(\psi^{-1}_{t,\cnu}(x)).
\end{equation*}
Based on the local perturbed solutions above, we define the discrete second variation of $E_\Gamma^0(u,B_R)$ as
\begin{equation*}
\Delta_{\cnu\cnu}^t E_\Gamma^0(u,B_R):= E_\Gamma^0(P_{t, \cnu}u, B_R) + E_\Gamma^0(P_{-t, \cnu}u, B_R) -2E_\Gamma^0(u, B_R).
\end{equation*}

\subsection{Interior BV estimate}
The interior BV estimate follows the spirit of \cite{CSV19}, which gives {a} quantitative flatness estimate in $B_1$ for a stable set in $B_R$. Let us first give the estimate of the discrete second variation of the energy $\Delta_{\cnu\cnu}^t E_\Gamma^0(u,B_R)$ in Lemma \ref{Dis-second} and  an identity for the nonlocal energy $\cE$ in Lemma \ref{identity}.
\begin{lem}\label{Dis-second}
Let $\bar{K}$ be the kernel in \eqref{B_energy} satisfying properties in Proposition \ref{prop_kernel2}. Then the discrete second variation of the energy $\Delta_{\cnu\cnu}^t E_\Gamma^0(u,B_R)$ satisfies the estimate
\begin{equation*}
\Delta_{\cnu\cnu}^t E_\Gamma^0(u,B_R)\leq C\frac{t^2}{R^2} \cE(u,B_R) \quad \text{for any }R\geq 1,
\end{equation*}
where $C$ is a constant.
\end{lem}
The proof of this lemma is given by \cite[Lemma 2.1]{CSV19} (see also \cite[Lemma 2.1]{FS20} and \cite[Lemma 3.2]{Gui19}). {Recall} {that} $\bar{K}$ {satisfies} properties (i)-(iii) in Proposition \ref{prop_kernel2}.

Next, {we recall an} identity for nonlocal energy, which is originally introduced in \cite{PSV13} and crucially used in \cite{CSV19} for the interior BV estimate. Note that it does not depend on exact formulas of the kernel $\bar{K}$ as long as the integrals are well-defined. In the remaining context, $f_+(x):=\max\{f(x),0\}$ and $f_-(x):=-\min\{f(x),0\}$.
\begin{lem}\label{identity}
Let $u,v$ be any measurable functions such that $\cE(u,B_R)<\8$ and $\cE(v,B_R)<\8$.  Then we have
\begin{equation*}
\begin{aligned}
& \cE(u, B_R)+ \cE(v,B_R)\\
=&\cE(\min\{u,v\}, B_R)+ \cE(\max\{u,v\},B_R) + 2 \iint_{\bR^2\times \bR^2\backslash B_R^c\times B_R^c} (v-u)_+(x) (v-u)_-(y) \bar{K}(x-y) \ud x \ud y,
\end{aligned}
\end{equation*}
where $\bar{K}$ is the kernel associated with the nonlocal energy $\cE$.
\end{lem}

Now we are ready to give the interior BV estimate for stable {solutions} in Definition \ref{def_stable}. The proof is similar to \cite[Lemma 3.6]{Gui19} and \cite[Lemma 2.2]{FS20} due to the properties of the {kernel} $\bar{K}$ in Proposition \ref{prop_kernel2}. We include the proof for completeness.
\begin{lem}\label{BV}
Let $|u|\leq M$ be a bounded stable solution to \eqref{1D-maineq} satisfying Definition \ref{def_stable}. Then there exists a constant $C(\beta,M)$ depending only on $\beta$ and $M$ such that for any $R\geq 1$,
\begin{align}
\left( \int_{B_{\frac12}}(\pt_{\cnu} u(x))_+ \ud x \right)\left( \int_{B_{\frac12}}(\pt_{\cnu} u(y))_- \ud y \right)\leq C(\beta,M) \frac{\cE(u, B_R)}{R^2},\label{BV1}\\
\int_{B_{\frac12}} |\nabla u(x)| \ud x \leq C(\beta,M) (1+\sqrt{\cE(u,B_1)}).\label{BV2}
\end{align}
\end{lem}
\begin{proof}
Step 1. Proof of \eqref{BV1}. Denote
$$u_M:= \max\{P_{t,\cnu} u, u\}\quad \text{and}\quad u_m:= \min\{P_{t, \cnu} u, u\}.
$$
Then by the identity in Lemma \ref{identity}, we have for $R\geq 1$,
\begin{equation}\label{bv-tm1}
\begin{aligned}
&\cE(u_m, B_R)+ \cE(u_M,B_R) + 2 \int_{B_{\frac12}}\int_{B_{\frac12}} (u(x-t\cnu)-u(x))_+ (u(y-t\cnu)-u(y))_- \bar{K}(x-y) \ud x \ud y\\
\leq & \cE(u, B_R)+ \cE(P_{t,\cnu} u,B_R),
\end{aligned}
\end{equation}
where we used $P_{t,\cnu} u(x) = u(x-t\cnu)$ for $x\in B_{\frac12}$ and $|t|$ small enough.
Moreover, for the local term $F$ in total energy, we always have
\begin{equation}\label{bv-tm2}
F(u_m, B_R)+ F(u_M,B_R) = F(P_{t, \cnu}(u), B_R)+ F(u, B_R).
\end{equation}
Since $|x-y|<1$ for $x,y\in B_{\frac12}$ and
$$
0<\frac{c_\beta}{|x-y|^3}\leq \bar{K}(x-y)
$$
from Proposition \ref{prop_kernel2}, \eqref{bv-tm1} and \eqref{bv-tm2} yield
\begin{equation*}
\begin{aligned}
&E_\Gamma^0(u_m, B_R)+ E_\Gamma^0(u_M,B_R) + C(\beta) \int_{B_{\frac12}}\int_{B_{\frac12}} (u(x-t\cnu)-u(x))_+ (u(y-t\cnu)-u(y))_-  \ud x \ud y\\
\leq & E_\Gamma^0(u, B_R)+ E_\Gamma^0(P_{t,\cnu} u,B_R).
\end{aligned}
\end{equation*}
Then by the stability of $u$ and  Lemma \ref{Dis-second}, we have
\begin{equation}\label{bv-tm4}
\begin{aligned}
&C(\beta) \int_{B_{\frac12}}\int_{B_{\frac12}} (u(x-t\cnu)-u(x))_+ (u(y-t\cnu)-u(y))_-  \ud x \ud y\\
\leq & \Delta_{\cnu\cnu}^t E_\Gamma^0(u, B_R)- \left[E_\Gamma^0(u_m, B_R)+E_\Gamma^0(u_M,B_R)+E_\Gamma^0(P_{-t,\cnu}u, B_R)-3E_\Gamma^0(u, B_R)\right]\\
\leq &\Delta_{\cnu\cnu}^t E_\Gamma^0(u, B_R) + o(t^2) \leq C(\beta)\frac{t^2}{R^2}\cE(u, B_R).
\end{aligned}
\end{equation}
Here in the second inequality, we used the fact that the second variation of $E^0_\Gamma$ is nonnegative, which implies
$$[E_\Gamma^0(u_m, B_R)-E_\Gamma^0(u, B_R)]+[E_\Gamma^0(u_M,B_R)-E_\Gamma^0(u, B_R)]+[E_\Gamma^0(P_{-t,\cnu}u, B_R)-E_\Gamma^0(u, B_R)]\geq -o(t^2).$$
By dividing $t^2$ in \eqref{bv-tm4} and taking $t\to 0$, we conclude \eqref{BV1}.

Step 2. Proof of \eqref{BV2}. Denote
$$
A^\pm:= \int_{B_{\frac12}} (\pt_{\cnu} u(x))_\pm \ud x.
$$
Then \eqref{BV1} gives
\begin{equation*}
\min\{A^+, A^-\} \leq \frac{C(\beta)}{R}\sqrt{\cE(u, B_R)}.
\end{equation*}
Thus we have
\begin{equation*}
\begin{aligned}
&\int_{B_{\frac12}}|\pt_{\cnu} u| \ud x
=& A^+ + A^- = |A^+-A^-|+ 2\min\{A^+, A^-\} \leq C(\beta,M)(1+\sqrt{\cE(u,B_1)}),
\end{aligned}
\end{equation*}
where we used
$$
|A^+ - A^-|=|\int_{B_{\frac12}} \pt_{\cnu} u(x) \ud x|\leq \int_{\pt B_{\frac12}} |u \cnu\cdot n_{\pt B_{\frac12}}|\leq C(M)
$$
due to boundedness of $u$. Therefore we obtain \eqref{BV2} since $|\nabla u|\leq |\pt_{1} u|+ |\pt_{2} u|.$
\end{proof}

\subsection{Energy estimates in any balls}
In this subsection, we will prove the energy estimate in any balls by {combining} the interior BV estimate in Lemma \ref{BV} and a sharp interpolation inequality for the nonlocal energy $\cE$ below.
\begin{lem}\label{interpolation}
Let $|u|\leq M$ be a bounded function. Assume that $u$ is Lipschitz in $B_2$ with $L_0:= \max\{2, \|\nabla u\|_{L^\8(B_2)}\}$. Then there exists a constant $C(M)$ depends only on $M$ such that
\begin{equation*}
\cE(u, B_1)\leq \iint_{\bR^2\times \bR^2\backslash B_1^c\times B_1^c} \frac{|u(x)-u(y)|^2}{|x-y|^3} \ud x \ud y \leq C(M) \log L_0 (1+\int_{B_2}|\nabla u| \ud x).
\end{equation*}
\end{lem}
This lemma is proved in \cite[Lemma 3.1]{FS20} for the kernel $\frac{1}{|x|^3}$, and we conclude this lemma for {the} kernel $\bar{K}$ since $\bar{K}(x)\leq \frac{1}{|x|^3}$ due to Proposition \ref{prop_kernel2}.

With this sharp interpolation lemma and the interior BV estimate in Lemma \ref{BV}, we are ready to obtain the energy estimates in any balls below.
\begin{prop}\label{energy-es}
Let $|u|\leq M$ be a bounded stable solution to \eqref{1D-maineq} satisfying Definition \ref{def_stable}. Assume {that} $W$ satisfies \eqref{potential} and  $L_*:=\max\{2, \|W\|_{C^{2,\alpha}_b(\bR)}\}$. Then there exists constant $C(\beta, M, L_*)$ depending only on $\beta, M$ and $L$ such that for any $B_R\subset \bR^2$ and $R\geq 1$,
\begin{equation}\label{uniform-e}
\int_{B_R} |\nabla u| \ud x \leq C(\beta, M,L_*) R \log (L_*R ), \qquad \cE(u, B_R)\leq C(\beta, M,L_*) R \log^2(L_*R).
\end{equation}
\end{prop}
\begin{proof}
First, by interior regularity estimate for $\cL$, for $L_1:=\|W\|_{C^{2,\alpha}_b(\bR)}$, we have
\begin{equation*}
\|\nabla u\|_{L^\8(B_2)}\leq C L_1.
\end{equation*}
Denote $L_2:=\max\{2, CL_1\}$. Then combining \eqref{BV2} and Lemma \ref{interpolation}, we have
\begin{equation}\label{tm419}
\begin{aligned}
\int_{B_{\frac12}} |\nabla u(x)| \ud x &\leq C(\beta,M) \left(1+\sqrt{C(M)\log L_2 \left(1+\int_{B_2}|\nabla u| \ud x\right)}\right).\\
&\leq \frac{C(\beta,M) \log L_2}{\delta} + \delta \int_{B_2}|\nabla u| \ud x,
\end{aligned}
\end{equation}
where we used Young's inequality.

Second, we prove {a} uniform bound by {a} scaling argument {and a standard iteration argument}. For any $z$, choose $\rho<1$ such that $B_\rho(z)\subset B_1$ and $\tilde{u}(x):=u(z+\frac{\rho}{2} x)$. Notice that $$
\bar{K}(\frac{2x}{\rho}) = (\rho/2)^3 \bar{K}(x)
$$
due to Proposition \ref{prop_kernel2}. Then $\tilde{u}$ satisfies \eqref{1D-maineq} with $W$ replaced by $\frac{\rho}{2} W$. Therefore, \eqref{tm419} still holds, i.e.,
\begin{equation*}
\int_{B_{\frac12}} |\nabla \tilde{u}(x)| \ud x \leq \frac{C(\beta,M) \log 2L_2}{\delta} + \delta \int_{B_2}|\nabla \tilde{u}| \ud x,
\end{equation*}
which is equivalent to
\begin{equation*}
\frac{1}{\rho} \int_{B_{\frac{\rho}{4}}(z)}|\nabla u| \ud x \leq \frac{C(\beta,M)\log L_2}{\delta} + \frac{\delta}{\rho} \int_{B_\rho(z)}|\nabla u| \ud x.
\end{equation*}
Then by 
{a standard iteration argument}, one obtain
\begin{equation*}
\int_{B_{\frac12}}|\nabla u| \ud x \leq C(\beta, M)\log L_2.
\end{equation*}
By the same scaling argument with $u_R(x):=u(z+2R x)$, one can obtain
\begin{equation*}
\int_{B_R{(z)}} |\nabla u| \ud x \leq C(\beta, M) R \log (CRL_2) \quad \text{ for any }R\geq 1.
\end{equation*}
Moreover by Lemma \ref{interpolation} and scaling argument, we also have
\begin{equation*}
\cE(u, B_R)\leq C(\beta, M) R \log^2(CRL_2) \quad \text{ for any }R\geq 1.
\end{equation*}
Therefore, we conclude \eqref{uniform-e}.
\end{proof}

\subsection{1D profile conclusion}
In this subsection, we are in the position to state and prove that any bounded stable solution to \eqref{1D-maineq} {has} a 1D monotone profile.

Now we give the main theorem in this section, which corresponds to the flatness result for 2D minimal surface with fractional anisotropic perimeters.
\begin{thm}\label{thm-1d}
Let $\beta=1-\nu\in(\frac23, \frac32)$. Assume that $|u|\leq M$ is a bounded stable solution to \eqref{1D-maineq} and $W$ satisfies \eqref{potential}. Then $u$ {has} a 1D monotone profile and $|u|\leq 1$. As a consequence, {any} bounded stable solution to \eqref{1D-neq} {also has a} 1D monotone profile and $|u|\leq 1$.  Moreover, the solution to \eqref{1D-neq} can be characterized as $u(x)=\phi(e\cdot x)$ for any $e:=(\cos \alpha, \sin \alpha)\in S^1$ with $\alpha \in (-\frac{\pi}{2},\frac{\pi}{2})$, where $\phi$  is the unique (up to translations) solution to 1D problem
\begin{equation}\label{1Dphi}
\begin{aligned}
(-\Delta)^{\frac12}\phi(x_1) &= - (\beta \cos^2 \alpha+ \sin^2\alpha) W'(\phi(x_1)), \, x_1\in\mathbb{R}\\
\lim_{x_1\to \pm\8}\phi(x_1)&=\pm 1.
\end{aligned}
\end{equation}
\end{thm}
\begin{proof}
{Combining} the uniform energy estimate \eqref{uniform-e} with the interior BV estimate \eqref{BV1}, taking $R\to +\8$, we know {that}
\begin{equation*}
\left( \int_{B_{\frac12}}(\pt_{\cnu} u(x))_+ \ud x \right)\left( \int_{B_{\frac12}}(\pt_{\cnu} u(y))_- \ud y \right)=0.
\end{equation*}
Since this is true for any direction $\cnu\in S^1$ and any half ball in $\bR^2$, we have
\begin{equation*}
\pt_{\cnu} u \geq 0  \text{ in }\bR^2 \quad \text{ or } \quad \pt_{\cnu} u \leq 0 \text{ in } \bR^2 \qquad \text{ for any }\nu\in S^1,
\end{equation*}
which yields the conclusion {that} $u$ {has} a 1D monotone profile.

Next, we prove that $u$ is given by $\phi(e\cdot x)$ and $\phi$ is the solution to the 1D problem \eqref{1Dphi}.

Let the direction $e$  be $e=(\cos \alpha, \sin \alpha)$ for some $\alpha$. Due to the far field boundary condition \eqref{BC_far}, we consider only the case $\alpha\in (-\frac{\pi}{2},\frac \pi 2)$.
Define the rotation matrix
\begin{equation*}
R:=\left[ \begin{array}{cc}
\cos \alpha & -\sin \alpha \\
\sin \alpha & \cos \alpha
\end{array}  \right] \quad \text{with } \text{det} R = 1,\, R^{-1}=R^T.
\end{equation*}
Define {the} new coordinates under the rotation matrix as
\begin{equation*}
\left( \begin{array}{c}
\bar{x}_1\\
\bar{x}_2
\end{array} \right):= R^T \left( \begin{array}{c}
{x}_1\\
{x}_2
\end{array} \right)= \left[ \begin{array}{cc}
\cos \alpha & \sin \alpha \\
-\sin \alpha & \cos \alpha
\end{array}  \right] \left( \begin{array}{c}
{x}_1\\
{x}_2
\end{array} \right), \quad
\left( \begin{array}{c}
\bar{k}_1\\
\bar{k}_2
\end{array} \right):= R^T \left( \begin{array}{c}
{k}_1\\
{k}_2
\end{array} \right).
\end{equation*}
Then we find $\phi$ such that $u(x)=\phi(e\cdot x)$ satisfies  \eqref{original}, i.e.,
\begin{equation*}
 -\F(W'(u))(k) = -\F(W'(\phi(e\cdot x))) =\frac{|k|^3}{\beta k_1^2 + k_2^2} \F(u)(k) = \frac{|k|^3}{\beta k_1^2 + k_2^2} \F(\phi(e\cdot x))(k).
\end{equation*}
This, together with the property that Fourier transform commutes with rotations, implies
\begin{equation*}
\begin{aligned}
&
\frac{|k|^3}{\beta k_1^2 + k_2^2} \F(\phi(\bar{x}_1))(\bar{k}_1) \delta(\bar{k}_2)\\
=& -\F(W'(\phi(e\cdot x))) = -\F(W'(\phi))(\bar{k}_1) \delta(\bar{k}_2).
\end{aligned}
\end{equation*}
Therefore, $\phi(x_1)$ is the solution to \eqref{1Dphi}, or equivalently
\begin{equation*}
|k_1| \hat{\phi}(k_1) = -\tilde{\beta} \F(W'(\phi)) (k_1)
\end{equation*}
with $\tilde{\beta}$ satisfying
\begin{equation*}
\frac{|\bar{k}_1|}{\tilde{\beta}} = \frac{|k|^3}{\beta k_1^2 + k_2^2} , \quad \bar{k}_2= -k_1 \sin \alpha+ k_2 \cos\alpha =0.
\end{equation*}
Then elementary calculations yield
\begin{equation*}
\tilde{\beta}=\beta \cos^2 \alpha+ \sin^2\alpha.
\end{equation*}
From \cite{CS05, PSV13}, the solution $\phi$ to \eqref{1Dphi}  is unique (up to translations), bounded, increasing from $-1$ to $1$, and a local minimizer of the isotropic nonlocal energy
$$
E^i_\Gamma={\frac 1 2}\int_{\bR} u (-\Delta)^{\frac12} u \ud x +\tilde{\beta} \int_{\bR} W(u) \ud x.
$$
Thus the second local variation of $E_\Gamma^i$ is nonnegative; see also \cite{GL19} for  the positivity of the linearized operator $(-\Delta)^\frac12+ \tilde{\beta} W''(\phi)I$.   Therefore, $u(x)=\phi(e\cdot x)$   characterizes the bounded stable solutions to \eqref{1D-neq}.
\end{proof}

\begin{rem}
It is easy to verify that the local minimizer of {the} energy $E^0_\Gamma$ is a bounded stable solution to \eqref{1D-maineq}. From the proof of \cite[Remark 1.4]{Gui19} and Theorem \ref{thm-1d}, one also {knows that any} bounded stable solution for $\frac23<\beta<\frac32$ {has a} 1D monotone profile and thus  a local minimizer. That is to say, for \eqref{1D-maineq} (also \eqref{1D-neq}) with $\frac23<\beta<\frac32$, bounded stable solutions and local minimizers are {the} same set and both are 1D monotone.
\end{rem}

\begin{rem} Let $\nu=1-\beta\in(-\frac12, \frac13)$. From the solution $u_1$ to \eqref{maineq1D}, one can further solve the other two components $u_2, u_3$ by \eqref{u13} and the elastic extension \cite{GLLX19} based on the Dirichlet to Neumann map. Finally,  the stable solution to the full system \eqref{maineq} is {completely} solved.
\end{rem}


\appendix

\section{Derivation of Euler-Lagrange equation}\label{appA}
\begin{proof}[Proof of Lemma \ref{Lem2.2}]
From Definition \ref{minimizer} of {local minimizers}, we calculate the variation of {the} energy in terms of a perturbation with compact support in an arbitrary ball ${B_R}$.
For any $ \mathbf v\in C^\infty(B_R\backslash \Gamma)$ such that $ \mathbf v$ has compact support in ${B_R}$ and satisfies \eqref{bcphi}, we consider the perturbation $\delta \mathbf v$, where $\delta$ is a small real number. We denote  $\varepsilon:=\varepsilon(\mathbf u)$, $\sigma:=\sigma(\mathbf u)$ and $\varepsilon_1:=\varepsilon(\mathbf v)$, $\sigma_1:=\sigma(\mathbf v)$.
Then we have that
\begin{equation*}
\begin{aligned}
~&\lim_{\delta\to 0}\frac{1}{\delta} (E(\mathbf u+\delta \mathbf v)-E(\mathbf u))\\
=& \int_{B_R\backslash \Gamma}\frac{1}{2}(\sigma_1:\varepsilon+ \sigma:\varepsilon_1)\ud x  +\int_{{B_R}\cap\Gamma} \pt_{1}W(u_1^+, u_2^+)v_1^++ \pt_{2} W(u_1^+, u_2^+) v_2^+ \ud \Gamma\\
=&\int_{{B_R}\backslash \Gamma}\sigma:\varepsilon_1\ud x  +\int_{{B_R}\cap\Gamma}  \pt_{1}W(u_1^+, u_2^+)v_1^++ \pt_{2} W(u_1^+, u_2^+) v_2^+ \ud\Gamma\\
=&\int_{B_R\backslash \Gamma}\sigma:\nabla\mathbf v\ud x  +\int_{B_R\cap\Gamma} \pt_{1}W(u_1^+, u_2^+)v_1^++ \pt_{2} W(u_1^+, u_2^+) v_2^+ \ud\Gamma\\
=&-\int_{B_R\backslash \Gamma}\partial_j\sigma_{ij} v_i\ud x +\int_{B_R\cap\Gamma}\sigma_{ij}^+ n_j^+ v_i^+ \ud \Gamma \\
&\quad + \int_{B_R\cap\Gamma}\sigma_{ij}^- n_j^- v_i^- \ud \Gamma
 +\int_{B_R\cap\Gamma}\pt_{1}W(u_1^+, u_2^+)v_1^++ \pt_{2} W(u_1^+, u_2^+) v_2^+ \ud\Gamma \geq 0,
\end{aligned}
\end{equation*}
where we used the property that $\sigma$  and $\nabla \cdot \sigma$ are locally integrable in $\{x_3>0\}\cup\{x_3<0\}$ when carrying out the integration by parts, and
 the outer normal vector of the boundary $\Gamma$ is
 $\mathbf n^+$ (resp. $\mathbf n^-$) for the upper  {half-plane} (resp. lower half-plane). Similarly, taking perturbation as $-\mathbf v$, we have
 \begin{equation*}
\begin{aligned}
~&\lim_{\delta\to 0}\frac{1}{\delta} (E(\mathbf u-\delta \mathbf v)-E(\mathbf u))\\
=&\int_{B_R\backslash \Gamma}\partial_j\sigma_{ij} v_i\ud x -\int_{B_R\cap\Gamma}\sigma_{ij}^+ n_j^+ v_i^+ \ud \Gamma \\
&\quad - \int_{B_R\cap\Gamma}\sigma_{ij}^- n_j^- v_i^- \ud \Gamma
 -\int_{B_R\cap\Gamma} \pt_{1}W(u_1^+, u_2^+)v_1^++ \pt_{2} W(u_1^+, u_2^+) v_2^+ \geq 0.
\end{aligned}
\end{equation*}
 Hence
 \begin{align*}
&-\int_{B_R\backslash \Gamma}\partial_j\sigma_{ij} v_i\ud x +\int_{B_R\cap\Gamma}\sigma_{ij}^+ n_j^+ v_i^+ \ud \Gamma \\
&\quad + \int_{B_R\cap\Gamma}\sigma_{ij}^- n_j^- v_i^- \ud x \ud z
 +\int_{B_R\cap\Gamma}\pt_{1}W(u_1^+, u_2^+)v_1^++ \pt_{2} W(u_1^+, u_2^+) v_2^+ \ud\Gamma= 0.
 \end{align*}
  Noticing that $\mathbf n^+=(0,0,-1)$ and $\mathbf n^-=(0,0,1)$, we have
 \begin{equation*}
 \begin{aligned}
&\int_{B_R\cap\Gamma}\sigma_{ij}^+ n_j^+ v_i^+ \ud \Gamma + \int_{B_R\cap\Gamma}\sigma_{ij}^- n_j^- v_i^- \ud \Gamma \\
=& \int_{B_R\cap\Gamma}-\sigma_{33}^+  v_3^+ \ud x \ud z+ \int_{B_R\cap\Gamma}\sigma_{33}^-  v_3^- \ud \Gamma+ \int_{B_R\cap\Gamma}-\sigma_{13}^+  v_1^+ \ud \Gamma+ \int_{B_R\cap\Gamma}\sigma_{13}^-  v_1^- \ud \Gamma\\
&+  \int_{B_R\cap\Gamma}-\sigma_{23}^+  v_2^+ \ud \Gamma+ \int_{B_R\cap\Gamma}\sigma_{23}^-  v_2^- \ud \Gamma.
 \end{aligned}
 \end{equation*}
Recall that $v_1^+=-v_1^-$, $v_3^+=v_3^-$ and $v_2^+=-v_2^-$. Hence due to the arbitrariness of $R$, we conclude that the minimizer  $\mathbf u$ must  satisfy
\begin{equation*}
\begin{aligned}
&\int_{\Gamma}\left[\sigma_{13}^+ + \sigma_{13}^- -\pt_{1}W (u_1^+, u_2^+)\right] v_1^+ \ud \Gamma=0,\\
&\int_{\Gamma}\left[\sigma_{23}^+ + \sigma_{23}^- -\pt_{2}W (u_1^+, u_2^+)\right] v_2^+ \ud \Gamma=0,\\
&\int_\Gamma \left(\sigma_{33}^+-\sigma_{33}^-\right) v_3^+ \ud \Gamma =0,\\
&\int_{\mathbb{R}^2\backslash \Gamma} (\nabla\cdot \sigma) \cdot\mathbf v~ \ud x\ud y \ud z =0
\end{aligned}
\end{equation*}
for any $\mathbf v\in  C^\infty(B_R\backslash\Gamma)$ and $ \mathbf v$ has compact support in $B_R$,
 which leads to the Euler--Lagrange equation \eqref{maineq}. Here we write the equation $\nabla\cdot \sigma=0$ in $\mathbb{R}^2 \backslash \Gamma$ as the first equation of \eqref{maineq} in terms of the displacement $\mathbf u$, using the constitutive relation.
\end{proof}

\section{Dirichlet to Neumann map}\label{appB}
\begin{proof}[Proof of Lemma \ref{D2N}]
Step 1. We take the Fourier transform of the elastic equations in \eqref{maineq} with respect to $x_1,x_2$ and denote the corresponding Fourier variables as $k_1, k_2$.

Due to \eqref{BC_far}, $\mathbf{u}$ is unbounded and we take the Fourier transform for $\mathbf{u}$ with respect to $x_1, x_2$ by regarding them as tempered distributions.  For notation simplicity, denote the Fourier transforms {to be} $\hat{\mathbf{u}}$.
 Let $k=(k_1, k_2)$ and $|k|=\sqrt{k^2_1+k^2_2}$. We have
\begin{align}
(1-2\nu) \pt_{33} \hat{u}_1 - [(2-2\nu)k_1^2 + (1-2\nu)k_2^2] \hat{u}_1 + i k_1 \pt_3 \hat{u}_3 - k_1 k_2 \hat{u}_2 = 0,\label{u1}\\
(2-2\nu)\pt_{33}\hat{u}_3 - (1-2\nu)|k|^2\hat{u}_3 + i k_1 \pt_3 \hat{u}_1 + ik_2 \pt_3 \hat{u}_2 = 0, \label{u2}\\
(1-2\nu) \pt_{33} \hat{u}_2 - [(2-2\nu)k_2^2 + (1-2\nu)k_1^2] \hat{u}_2 + i k_2 \pt_3 \hat{u}_3 - k_1 k_2 \hat{u}_1 = 0.\label{u3}
\end{align}

We can first eliminate $\hat{u}_2$ using \eqref{u1}, then eliminate $\hat{u}_3$ and obtain the ODE for $\hat{u}_1$
\begin{equation*}
\pt_3^4 \hat{u}_1 - 2 |k|^2 \pt_3^2 \hat{u}_1+ |k|^4 \hat{u}_1=0.
\end{equation*}
Next we use this ODE for $\hat{u}_1$ to simplify \eqref{u1}, \eqref{u2}, and \eqref{u3} again and then eliminate $\hat{u}_1$ and $\hat{u}_2$ together. We obtain the ODE for $\hat{u}_3$
\begin{equation*}
\pt_3^4 \hat{u}_3 - 2|k|^2 \pt_3^2 \hat{u}_3 +|k|^4 \hat{u}_3 = 0.
\end{equation*}
{By the symmetry of} $\hat{u}_1$ and $\hat{u}_2$, we have {the} same ODE for  $\hat{u}_2$.

{We} seek for solutions whose derivatives have decay properties, which exclude exponentially growing solutions as $|x_3|\to +\8$.  Denote
 \begin{equation*}
 \hat{u}_1^- = (A^- + B^-|k| x_3)e^{|k|x_3}, \quad x_3<0,
 \end{equation*}
 where $A^-, B^-$ are constants to be determined. Similarly, denote
 \begin{equation*}
  \hat{u}_3^- = (C^- + D^-|k| x_3)e^{|k|x_3}, \quad  \hat{u}_2^- = (E^- + F^- x_3|k|)e^{|k|x_3},  \quad x_3<0,
 \end{equation*}
  where $C^-, D^-, E^-, F^-$ are constants to be determined. For $x_3>0$, we have another six constants $A^+, B^+, C^+, D^+, E^+, F^+$ to be determined and for $x_3>0$,
  \begin{equation*}
\begin{aligned}
 \hat{u}_1^+ &= (A^+ - B^+|k| x_3)e^{-|k|x_3}, \\
 \hat{u}_3^+ &= (C^+ - D^+ |k|x_3)e^{-|k|x_3}, \\
 \hat{u}_2^+ &= (E^+ - F^+ |k|x_3)e^{-|k|x_3}.
 \end{aligned}
  \end{equation*}

  Step 2. Given the Dirichlet values of $u_1$ and $u_2$, we express all the other constants by $A^\pm$ and $E^\pm$.

First, plugging $\hat{u}_1^-$, $\hat{u}_2^-$, and $\hat{u}_3^-$ into \eqref{u1}, we have
  \begin{equation*}
  (2-4\nu)|k|^2B^- - k_1^2 A^- + i k_1 (C^-|k|+ D^- |k|)- k_1 k_2 E^-=0
  \end{equation*}
{and}
  \begin{equation*}
  -k_{1}^{2} B^{-}+i k_{1} D^{-}|k|-k_{1} k_{2} F^{-}=0.
  \end{equation*}
  Plugging $\hat{u}_1^-$, $\hat{u}_2^-$, and $\hat{u}_3^-$ into \eqref{u2}, we have
  \begin{equation*}
|k|^{2} C^{-}+(4-4\nu)|k|^{2} D^{-}+i k_{1}|k| A^{-}+i k_{1}|k| B^{-}+i k_{2}|k| E^{-}+i k_{2}|k| F^{-}=0
  \end{equation*}
  and
  \begin{equation*}
|k|^2 D^- + i k_1 |k|B^- + i k_2 |k|F^- =0.
  \end{equation*}
  Plugging $\hat{u}_1^-$, $\hat{u}_2^-$, and $\hat{u}_3^-$ into \eqref{u3}, we have
  \begin{equation*}
  (2-4\nu)|k|^2F^- - k_2^2 E^- + i k_2 (C^-|k|+ D^- |k|)- k_1 k_2 A^-=0
  \end{equation*}
  and
  \begin{equation*}
  -k_{2}^{2} F^{-}+i k_{2} D^{-}|k|-k_{1} k_{2} B^{-}=0.
  \end{equation*}
  Simplifying these relations gives us
  \begin{equation*}
  \begin{aligned}&{B^{-}=\frac{i k_{1}}{|k|} D^{-}} ,\quad {F^-=\frac{i k_{2}}{|k|} D^{-}}, \\ &{-k_{1} A^--k_{2} E^{-}+i|k| C^-=(4 \nu-3) i|k| D^{-}.}\end{aligned}
  \end{equation*}
  Combining this with the boundary symmetry \eqref{BC}, we have
  \begin{equation*}
  A^+ = - A^-, \ \ B^+ = -B^-,\ \ C^+= C^-,\ \  D^+= D^-,\ \  E^+ = -E^-, \ \ F^+= -F^-.
  \end{equation*}
  Then by $\sigma_{33}^+= \sigma_{33}^-$ on $\Gamma$ in \eqref{maineq}, we further obtain  $C^-= (2\nu-1)D^-$. Therefore, all {the} other constants can be expressed {in terms of} $A^-$ and $E^-$. In particular, we conclude that $\sigma_{13}(x_1,x_2, 0^+)$ and $\sigma_{23}(x_1,x_2, 0^+)$ can be expressed as {in} \eqref{curve}.
\end{proof}

\bibliographystyle{plain}
\bibliography{bibpn1}

\end{document}